\documentclass{article}
\usepackage{arxiv}
\usepackage[version=4]{mhchem}
\usepackage{float}
\usepackage{booktabs}
\usepackage{algorithm}
\usepackage{algpseudocode}
\algrenewcommand\algorithmicrequire{\textbf{Inputs:}}
\algrenewcommand\algorithmicensure{\textbf{Outputs:}}
\usepackage{amsmath,amssymb,amsfonts,amsthm,mathtools}
\usepackage{a4wide}
\usepackage{xcolor}
\usepackage{verbatim}
\usepackage{epsfig,subfigure,epstopdf}
\usepackage[]{graphics}
\usepackage{graphicx,inputenc}
\usepackage{subfigure}
\usepackage{caption}
\usepackage{pictex}
\usepackage{wrapfig}
\usepackage{multirow}
\usepackage{placeins}
\usepackage{booktabs}
\usepackage{tabularx}
\usepackage{array}
\newcolumntype{Y}{>{\raggedright\arraybackslash}X}

\newcommand{\numtablestyle}{%
  \small
  \renewcommand{\arraystretch}{1.18}%
  \setlength{\tabcolsep}{6pt}%
}

\usepackage[T1]{fontenc}
\usepackage{authblk}

\definecolor{azure}{rgb}{0.0, 0.5, 1.0}
\definecolor{darkred}{rgb}{0.8, 0.0, 0.0}
\definecolor{DarkBlueNavy}{RGB}{0, 0, 128}
\definecolor{nicegreen}{RGB}{76, 175, 80}

\usepackage{fancyhdr}
\renewcommand{\shorttitle}{%
  \footnotesize\shortstack{Regularized Reconstruction of Dynamic Sources\\from Sparse Boundary Flux Data}%
}
\AtBeginDocument{\setlength{\headheight}{24pt}}

\newtheoremstyle{exampstyle}
  {\topsep} 
  {\topsep} 
  {\itshape} 
  {} 
  {\bfseries} 
  {.} 
  {.5em} 
  {} 

\theoremstyle{exampstyle}
\numberwithin{equation}{section}
\newtheorem{theorem}{Theorem}[section]
\newtheorem{definition}[theorem]{Definition}
\newtheorem{lemma}[theorem]{Lemma}

\newtheorem{remark}[theorem]{Remark}

\newtheorem{proposition}[theorem]{Proposition}

\allowdisplaybreaks
\usepackage{parskip}
\usepackage[nocompress]{cite}
\usepackage{hyperref}
\hypersetup{
colorlinks,
linkcolor={nicegreen},
citecolor={azure},
urlcolor={DarkBlueNavy}
}
\let\oldref\ref
\renewcommand{\ref}[1]{(\oldref{#1})}  
\renewcommand{\eqref}[1]{(\oldref{#1})} 
\DeclareMathOperator*{\argmin}{arg\,min}

\newbox\boxaddrone \newbox\boxaddrtwo

\def\N+{n\in\mathbb{N}^{+}}

\def\1d{\mathcal{D}((-\Delta)^{\gamma+1})}

\newcommand{\R}{\mathbb R}
\newcommand{\Q}{\mathbb Q}

\newcommand{\Dcal}{\mathcal D}
\newcommand{\dd}{\,\mathrm d}

\newcommand{\norm}[1]{\left\lVert #1\right\rVert}
\newcommand{\abs}[1]{\left\lvert #1\right\rvert}
\newcommand{\ip}[2]{\left\langle #1,#2\right\rangle}

\begin{document}
\title{Regularized Sequential Reconstructions of Dynamic Nonsmooth Sources from Sparse Boundary Flux Data}
\date{}

\author[1]{Davide Bianchi\thanks{bianchid@mail.sysu.edu.cn}}
\author[1]{Tianrui Dai\thanks{daitr@mail.sysu.edu.cn}}
\author[2]{Hongqin Zhang\thanks{hqzhang1112@mail.iap.ac.cn}}
\author[3,4]{Shupeng Zhang\thanks{zhangshp8@mail.sysu.edu.cn}}
\author[1,5]{Zhidong Zhang\thanks{zhangzhidong@mail.sysu.edu.cn}}
\affil[1]{\normalsize{School of Mathematics (Zhuhai), Sun Yat-sen University, Zhuhai 519082, Guangdong, China}}
\affil[2]{\normalsize{State Key Laboratory of Earth System Numerical Modeling and Application, Institute of Atmospheric Physics, Chinese Academy of Sciences, Beijing, 100029, China}}
\affil[3]{\normalsize{Southern Marine Science and Engineering Guangdong Laboratory (Zhuhai), Sun Yat-sen University, Zhuhai 519082, Guangdong, China}}
\affil[4]{\normalsize{School of Atmospheric Sciences, Sun Yat-sen University, Zhuhai 519082, Guangdong, China}}
\affil[5]{\normalsize{Guangdong Province Key Laboratory of Computational Science, Sun Yat-sen University, Guangzhou 510000, Guangdong, China}}

\setcounter{tocdepth}{2}

\maketitle

\begin{abstract}
We consider an inverse source problem for the heat equation on the unit disk $\Omega\subset\mathbb R^2$ with flux measurements at finitely many boundary points. The source is piecewise constant in time on prescribed intervals and, on each interval, is the characteristic function of an unknown region with Lipschitz boundary. The measurements on each interval split into the contribution of the current source and a history term from earlier sources; subtracting the latter leaves a linear inverse problem for a forward operator $F_\tau$ that is densely defined and unbounded on $L^2(\Omega)$. We prove that $F_\tau$ is not closable, so the classical regularization theory for closed unbounded operators does not apply. We regularize instead with a Hilbert-scale penalty that compactifies it. The resulting Galerkin--Tikhonov reconstruction is sequential: each source is recovered from the flux observed up to the end of its interval and from the previously reconstructed sources. We prove its convergence and uniqueness, show that earlier reconstruction errors propagate into later data with exponential decay in time, and establish  stability estimates with respect to the number of sensors. A numerical experiment on noisy measurements shows the feasibility of the sequential approximations, refined by Fourier approximation and boundary correction by a neural network.
\end{abstract}



\section{Introduction}\label{section_intro}

We consider an inverse source problem for the model  equation
\begin{equation*}\tag{ME}\label{PDE_2d}
 \begin{cases}
  \begin{aligned}
    (\partial_t-\Delta)u(x,t)&=P^\dagger(x,t), &&(x,t)\in\Omega\times(0,T),\\
    u(x,t)&=0,&&(x,t)\in\partial\Omega\times(0,T)\cup \Omega\times\{0\},
  \end{aligned}
  \end{cases}
\end{equation*}
on the unit disk $\Omega\subset\mathbb R^2$. The unknown source $P^\dagger$ is piecewise constant in time and has the form
\begin{equation}\label{source}
  P^\dagger(x,t):=\sum_{k=1}^K p_k^\dagger(x)\chi_{I_k}(t),
  \qquad
  I_k:=[t_{k-1},t_k),\qquad
  p_k^\dagger:=\chi_{D_k},
\end{equation}
where the time mesh $0=t_0<t_1<\ldots<t_K=T$ is prescribed and the spatial regions $\{D_k\}_{k=1}^K$ are unknown Lipschitz regions, that is, open subsets of $\Omega$ with Lipschitz boundary (Definition~\ref{def:regions}). These regions are the actual unknowns of the problem; the sources $p_k^\dagger$ are their characteristic functions. For an integer $M\ge 2$, we choose pairwise distinct boundary sensors
\[
z_j=(\cos\theta_j,\sin\theta_j)\in\partial\Omega, \qquad j=1,\ldots,M,
\]
and denote the set of sensors by $\Gamma_M:=\{z_1,\ldots,z_M\}$. We assume that $\Gamma_M$ is \emph{admissible}, in the sense that there exist indices $a,b\in\{1,\ldots,M\}$ with
\[
\theta_a-\theta_b\notin\pi\Q.
\]
The measurements consist of the boundary flux data
\begin{equation}\label{data}
g(t):=
\left(
\partial_{\mathbf n}u(z_1,t),
\ldots,
\partial_{\mathbf n}u(z_M,t)
\right)^\top\in\R^M,\qquad 0<t<T,
\end{equation}
with $\mathbf n$ the outward unit normal to $\partial\Omega$. The inverse model problem we address is the following:
\begin{equation*}\tag{MP}\label{inverse_problem}
  \text{recover}\ (p_1^\dagger,\ldots,p_K^\dagger)\ \text{in \eqref{source} from the data \eqref{data}}.
\end{equation*}
We denote the full forward map by $\mathcal F$, so that $\mathcal F(p_1^\dagger,\ldots,p_K^\dagger)=g$. Recovering the indicator sources $p_k^\dagger$ in $L^2(\Omega)$ determines the regions $D_k$: thresholding at $\tfrac12$ gives $D_k$ up to a set of measure zero, and a Lipschitz region is determined by its characteristic function in $L^2(\Omega)$ (Lemma~\ref{lem:regions-determined}).

This formulation is motivated by applications in which direct measurements inside the domain are impractical or unsafe. In atmospheric science, for instance, Equation~\eqref{PDE_2d} models pollutant diffusion \cite{byun2006review,EganMahoney:1972}: $u$ is the pollutant concentration, the regions $\{D_k\}_{k=1}^K$ represent the emission regions active on the successive time intervals, and only boundary measurements such as \eqref{data} are available. Determining the locations of these sources amounts precisely to solving \eqref{inverse_problem}, which makes both its theoretical analysis and its numerical resolution of practical relevance.

Beyond this classical application, modern numerical methods for advection–diffusion and related diffusion-dominated equations have been successfully applied to a broad range of environmental and geophysical problems, including:
\begin{itemize}
    \item Urban air quality management, where CFD-coupled models simulate \ce{NO2} and \(\mathrm{PM}_{10}\) dispersion in street canyons ~\cite{lin2022modelling};
    \item  Nuclear emergency response, where Lagrangian particle models track radionuclide transport following accidents such as Fukushima ~\cite{bakels2024flexpart};
    \item Ecosystem monitoring, where steady-state solutions compute flux footprints for spatial attribution of greenhouse gas emissions from tower measurements ~\cite{schlutow2026boundary};
    \item Climate intervention studies, where stochastic dispersion models estimate aerosol plume spreading rates for marine cloud brightening ~\cite{mcmichael2024exploring}.
\end{itemize}
Closely related diffusion models also arise in hydrology and land-surface science:
\begin{itemize}
    \item the Richards equation, which governs water flow in partially saturated porous media, is central to hydrology, agriculture, environmental modeling, and waste management \cite{farthing2017numerical};
    \item lake schemes in land-surface models, such as CoLM-Lake, describe lake thermal dynamics in climate systems through heat diffusion, vertical mixing, phase transitions, and surface energy exchange \cite{dai2018lake}.
\end{itemize}

Despite their differing physical settings, these applications share the same
structural difficulty, which can be decomposed in three features: the source is confined to small regions, it changes over
time, and it must be inferred from a few point measurements at the boundary.

Each of these  features is reflected in the literature on inverse source problems
for parabolic equations, which has a long history. Representative contributions
include \cite{BeauchardCannarsaYamamoto:2014,ChaikovskiiLiubavinZhang:2023,
ChengLuYamamoto:2023,ImanuvilovLiuYamamoto:2024,WangZhangZhang:2023}. The last
feature in particular has driven a growing body of work on sparse-data settings,
in which the observation set is a finite subset of $\partial\Omega$, as is the
case for $\Gamma_M$ here. We refer to \cite{HelinLassasYlinenZhang:2020,
KianLiLiuYamamoto:2021,LiLiuShi:2023,LiYangZhangZhang:2021,RundellZhang:2018} and
the references therein. On the unit disk, uniqueness from finitely many boundary
points is available for a time-independent source \cite{HettlichRundell2001} and
for a separable source $p(x)q(t)$ \cite{RundellZhang2020}. These results underlie
the sampling-based reconstructions with moving sensors of
\cite{LinOuZhangZhang2024,GuZhangZhang2026}, where the unknown is a single static
region described by a finite-dimensional shape parameterization. Sources of the
semi-discrete form \eqref{source} have been treated in \cite{LinZhangZhang2022} on
general domains, but there the flux is required on a nonempty open subset of
$\partial\Omega$. Problem \eqref{inverse_problem} sits between the two: the source
is piecewise constant in time with a distinct unknown support on each interval,
and the data consist of finitely many fixed boundary sensors.

Building on this sparse-observation framework, the present work develops a sequential
reconstruction method for~\eqref{inverse_problem} and tests it numerically. The method
exploits a triangular structure of the forward map $\mathcal F$ with respect to the time
intervals: the data on $I_k=[t_{k-1},t_k)$ depend only on $p_1^\dagger,\ldots,p_k^\dagger$,
and the current source $p_k^\dagger$ is acted on by the same type of
operator. Once the observations on $I_k$ are collected, $p_k^\dagger$ is therefore
reconstructed from them and from the estimates of the earlier sources alone, without
measurements from later intervals. The coupled problem $\mathcal F(p_1^\dagger,\ldots,p_K^\dagger)=g$
is thus replaced by $K$ inverse problems governed by the same linear forward map, solved one after the other, each of which
must be regularized while the errors inherited from the previous steps are kept under control.

To make this precise, let $X:=L^2(\Omega)$ and let $\{(\lambda_n,\varphi_n)\}_{n\ge1}$ be the eigenpairs of the Dirichlet Laplacian $\mathcal L=-\Delta$ on $\Omega$, with $\{\varphi_n\}$ orthonormal in $X$ and $0<\lambda_1\le\lambda_2\le\cdots\to\infty$, and write $f_n:=\ip{f}{\varphi_n}_X$ for the spectral coefficients of $f\in X$. Set $\tau_k:=t_k-t_{k-1}$ and collect the data on the $k$-th interval in $v_k(t):=g(t_{k-1}+t)$, $0<t<\tau_k$, an element of $Y_{\tau_k}:=L^2(0,\tau_k;\R^M)$. Since $g$ is $\R^M$-valued, one entry per sensor, we collect the sensor coefficients of each eigenfunction in a vector,
\[
\mathbf c_n:=\bigl(c_{z_1,n},\ldots,c_{z_M,n}\bigr)^\top\in\R^M,
\qquad
c_{z,n}:=-\lambda_n^{-1}\partial_{\mathbf n}\varphi_n(z).
\]
All series below are then $\R^M$-valued, summed componentwise, and their $j$-th component is the flux at the sensor $z_j$. Finally, the sources are measured in the Hilbert scale generated by $\mathcal L$: for $s>0$, let
\[
L_s:=(I+\mathcal L)^{s/2},
\qquad
L_s\varphi_n=(1+\lambda_n)^{s/2}\varphi_n,
\qquad
\Dcal(L_s):=\Bigl\{f\in X:\sum_{n\ge1}(1+\lambda_n)^s|f_n|^2<\infty\Bigr\},
\]
so that $L_s$ is an isometric isomorphism from $(\Dcal(L_s),\|L_s\cdot\|_X)$ onto $X$ and $L_s^{-1}$ is bounded and compact on $X$. For $0<s<\tfrac12$ one has $\Dcal(L_s)=H^s(\Omega)$, which still contains the characteristic function of every Lipschitz region (Lemma~\ref{lem:Lsdomain}); such an $s$ is fixed throughout.

Expanding $u(\cdot,t)=\sum_nu_n(t)\varphi_n$, each coefficient solves $u_n'+\lambda_nu_n=\sum_ip_{i,n}^\dagger\chi_{I_i}$ with $u_n(0)=0$, by the homogeneous initial condition in \eqref{PDE_2d}, where $p_{i,n}^\dagger:=\ip{p_i^\dagger}{\varphi_n}_X$. Variation of constants on $I_k$ gives, for $0<t<\tau_k$,
\[
u_n(t_{k-1}+t)
=\frac{1}{\lambda_n}\Bigl[\underbrace{p_{k,n}^\dagger\,(1-e^{-\lambda_nt})}_{\text{current source}}
+\underbrace{\sum_{i<k}p_{i,n}^\dagger\,(1-e^{-\lambda_n\tau_i})\,e^{-\lambda_n(t+t_{k-1}-t_i)}}_{\text{earlier sources}}\Bigr].
\]
The first term is driven by the source active on $I_k$; each earlier source $p_i^\dagger$ contributes only what it injected during $I_i$, damped by the factor $e^{-\lambda_n(t_{k-1}-t_i)}$ of the elapsed time. Since $\partial_{\mathbf n}u(z,t)=-\sum_n\lambda_nu_n(t)\,c_{z,n}$, the flux splits in the same way, and the two kinds of terms define the operators
\begin{align*}
F_\tau:\Dcal(L_s)\subset X\to Y_\tau,\qquad (F_\tau f)(t)&:=\sum_{n\ge1}\mathbf c_n f_n\,(e^{-\lambda_nt}-1),
&&0<t<\tau,\\
H_{k,i}:X\to Y_{\tau_k},\qquad (H_{k,i}f)(t)&:=-\sum_{n\ge1}\mathbf c_n f_n\,(1-e^{-\lambda_n\tau_i})\,e^{-\lambda_n(t+t_{k-1}-t_i)},
&&0<t<\tau_k,
\end{align*}
the current-interval forward operator $F_\tau$ and, for $1\le i<k\le K$, the history operators $H_{k,i}$ (Definition~\ref{current-slab-and-history-op}). Both series converge absolutely and uniformly in time for $f\in\Dcal(L_s)$, in particular for the characteristic sources, and the series defining $H_{k,i}f$ converges in $Y_{\tau_k}$ for every $f\in X$. The choice of the domain $\Dcal(L_s)$ is discussed in Remark~\ref{rem:Ftau-domain}. In this notation the forward map is block lower triangular (Lemma~\ref{lem:history-decomposition}),
\begin{equation}\label{eq:intro-decomposition}
\begin{pmatrix} v_1\\ v_2\\ \vdots\\ v_K\end{pmatrix}
=
\begin{pmatrix}
F_{\tau_1} & & & \\
H_{2,1} & F_{\tau_2} & & \\
\vdots & & \ddots & \\
H_{K,1} & H_{K,2} & \cdots & F_{\tau_K}
\end{pmatrix}
\begin{pmatrix} p_1^\dagger\\ p_2^\dagger\\ \vdots\\ p_K^\dagger\end{pmatrix},
\qquad\text{i.e.}\qquad
v_k=\sum_{i<k}H_{k,i}p_i^\dagger+F_{\tau_k}p_k^\dagger .
\end{equation}
The diagonal blocks are all of the same type $F_\tau$, and the history blocks act only on sources that precede $p_k^\dagger$. Forward substitution in \eqref{eq:intro-decomposition} therefore yields, for $k=1,\ldots,K$, a problem for the current source alone,
\begin{equation}\label{eq:intro-shifted}
F_{\tau_k}p_k^\dagger=\widetilde v_k,
\qquad
\widetilde v_1:=v_1,
\qquad
\widetilde v_k:=v_k-\sum_{i<k}H_{k,i}p_i^\dagger\quad (k\ge2).
\end{equation}
 This is the sequential scheme analyzed below. Its analysis reduces to the study of the single linear operator $F_\tau$ and of the propagation of errors through the history operators $H_{k,i}$.

The  core of the paper studies the properties of the linear
problem~\eqref{eq:intro-shifted}, and its regularization. We prove that $F_\tau$, as an operator in $X$, is unbounded and \emph{not closable} (Lemma~\ref{lem:Ftau-unbounded} and Proposition~\ref{prop:F-unbounded}): there are sources $f_j\in\Dcal(L_s)$ with $\|f_j\|_X\to0$ whose fluxes $F_\tau f_j$ converge in $Y_\tau$ to a nonzero time-constant datum. Consequently, the classical resolvent-based regularization theory for
closed unbounded operators \cite{Groetsch1992spectral,groetsch2007stable} does not apply here, and we regularize~\eqref{eq:intro-shifted} by
\begin{equation}\label{eq:intro-tikhonov-Ls}
\widehat{p} = \argmin_{p\in\Dcal(L_s)}\ \Bigl\{\tfrac12\|F_\tau p-y\|_{Y_\tau}^2+\tfrac\alpha2\|L_sp\|_X^2\Bigr\},
\end{equation}
which is a generalized Tikhonov regularization  with  penalty term in a Hilbert scale; see for example \cite[Section~8.4]{EnglHankeNeubauer1996}.  
The  penalty does two things. First, it does not exclude the true sources, which belong to $\Dcal(L_s)$ by Lemma~\ref{lem:Lsdomain}. Second, it is exactly the norm in which $F_\tau$ is bounded, so that \eqref{eq:intro-tikhonov-Ls} is a Tikhonov problem for a bounded operator between Hilbert spaces, to which the classical theory applies.

By a change of variables $\phi:=L_sp$, since $L_s$ is an isometry from $(\Dcal(L_s),\|L_s\cdot\|_X)$ onto $X$, problem~\eqref{eq:intro-tikhonov-Ls} translates to the equivalent standard Tikhonov problem for the operator $\mathcal A_{\tau,s}:X\to Y_\tau$,
\begin{align}\label{eq:intro-tikhonov-phi}
&\widehat{\phi} = \argmin_{\phi\in X}\ \Bigl\{\tfrac12\|\mathcal A_{\tau,s}\phi-y\|_{Y_\tau}^2+\tfrac\alpha2\|\phi\|_X^2\Bigr\}, \\
&(\mathcal A_{\tau,s}\phi)(t):=(F_\tau L_s^{-1}\phi)(t)=\sum_{n\ge1}(1+\lambda_n)^{-s/2}\,\mathbf c_n\phi_n\,(e^{-\lambda_nt}-1), \qquad 0<t<\tau.\nonumber
\end{align}
The operator $\mathcal A_{\tau,s}$ is Hilbert--Schmidt  from $X$ to $Y_\tau$, and  injective (Propositions~\ref{prop:Ataus-HS} and~\ref{prop:injectivity}), so the classical theory for compact operators now applies. The change of variables is exact for the data as well: with $\phi_k^\dagger:=L_sp_k^\dagger\in X$, problem \eqref{eq:intro-shifted} becomes $\mathcal A_{\tau_k,s}\phi_k^\dagger=\widetilde v_k$, and injectivity yields, by forward substitution in \eqref{eq:intro-decomposition}, uniqueness for the full dynamic problem (Proposition~\ref{proposition_uniqueness}). 

The reconstruction is made by forward substitution in \eqref{eq:intro-decomposition} with one regularized problem \eqref{eq:intro-tikhonov-phi} per interval. Given noisy data 
\begin{equation*}
   v_k^\delta \qquad \mbox{with}\qquad   \|v_k^\delta-v_k\|_{Y_{\tau_k}}\leq\delta,
\end{equation*}
and the reconstructions $\widehat p_1^\delta,\ldots,\widehat p_{k-1}^\delta$ of the earlier sources, we set $\widetilde v_k^\delta:=v_k^\delta-\sum_{i<k}H_{k,i}\widehat p_i^\delta$ and compute, on the  Galerkin space $V_n:=\operatorname{span}\{\varphi_1,\ldots,\varphi_n\}$,
\begin{equation*}
\widehat\phi_k^\delta
:=\operatorname*{argmin}_{\phi\in V_{n_k}}
\Bigl\{
\tfrac12\|\mathcal A_{\tau_k,s}\phi-\widetilde v_k^\delta\|_{Y_{\tau_k}}^2
+\tfrac{\alpha_k}2\|\phi\|_X^2
\Bigr\},
\qquad
\widehat p_k^\delta:=L_s^{-1}\widehat\phi_k^\delta .
\end{equation*}
Since $L_s^{-1}V_n=V_n$, this is exactly the Hilbert-scale problem \eqref{eq:intro-tikhonov-Ls} for $F_{\tau_k}$ on $V_{n_k}$. The method is stated in Definition~\ref{def:sequential-reconstruction} and summarized in Algorithm~\ref{alg:sequential} in Section~\ref{subsec:mainthm}. Our main results are the following.
\begin{itemize}
\item \emph{Convergence.} $\|\widehat p_k^\delta-p_k^\dagger\|_X\to0$ as $\delta\to0$ for every $k$, provided $\alpha_k(\delta)$ and $n_k(\delta)$ are chosen consistently with the effective error $\eta_k^\delta:=\|\widetilde v_k^\delta-F_{\tau_k}p_k^\dagger\|_{Y_{\tau_k}}$ of the shifted data (Theorem~\ref{thm:dynamic-convergence}). Thresholding then recovers the regions $D_k$ in measure.
\item \emph{Error propagation.} Errors on earlier intervals enter later data through the history operators, with exponential decay (Section~\ref{section_quantitative}):
\[
\eta_k^\delta\le\delta+C_H\sum_{i<k}e^{-\lambda_1(t_{k-1}-t_i)}\,\|\widehat p_i^\delta-p_i^\dagger\|_X .
\]
\item \emph{Sensor stability.} For nested sensor sets, the singular values of $\mathcal A_{\tau,s}$ restricted to $V_n$ are nondecreasing as sensors are added, so the finite-dimensional inverse stability constants are nonincreasing (Section~\ref{section_sensor_stability}).
\end{itemize}

The rest of the paper is organized as follows. Section~\ref{section_pre} fixes the notation, recalls the function spaces and the operator-theoretic facts used later. Section~\ref{section_main} introduces the block triangular splitting of the forward operator into current and past contributions and the compactified formulation, proves injectivity and finite-window uniqueness for the dynamic inverse problem, and develops the sequential Galerkin--Tikhonov reconstruction together with its convergence analysis. Section~\ref{section_stability} estimates the propagation of reconstruction errors into later data and establishes inverse stability estimates with respect to the number of sensors. Section~\ref{section_num} presents the numerical experiments.

\section{Notation and preliminaries}\label{section_pre}
Throughout, $\Omega\subset\R^2$ is the open unit disk, $X:=L^2(\Omega)$ with inner product $\ip{\cdot}{\cdot}_X$ and norm $\|\cdot\|_X$. For a measurable set $E\subset\R^2$, $|E|$ is its Lebesgue measure, $\chi_E$ its characteristic function, and $E\triangle D:=(E\setminus D)\cup(D\setminus E)$ the symmetric difference. Points of $\Omega$ are written in polar coordinates $x=(r\cos\theta,r\sin\theta)$, every $z\in\partial\Omega$ is written as $z=(\cos\theta_z,\sin\theta_z)$, and $\partial_{\mathbf n}=\partial_r$ is the outward normal derivative on $\partial\Omega$. On $\R^M$ we use the Euclidean norm $\|\cdot\|_{\R^M}$ and inner product $a\cdot b$. For $\tau>0$ the data space on a time window of length $\tau$ is
\begin{equation}\label{eq:Ytau}
Y_\tau:=L^2(0,\tau;\R^M),
\qquad
\|v\|_{Y_\tau}^2:=\int_0^\tau\|v(t)\|_{\R^M}^2\dd t .
\end{equation}
Exact quantities carry a dagger (i.e. $p_k^\dagger$, $\phi_k^\dagger$), noisy data a superscript $\delta$ (i.e. $v_k^\delta$), and reconstructions a hat (i.e. $\widehat p_k^\delta$).

\subsection{The unknown regions}\label{subsec:regions}
The quantities to be recovered in \eqref{inverse_problem} are the regions $D_1,\ldots,D_K$ of \eqref{source}; the sources $p_k^\dagger=\chi_{D_k}$ are their characteristic functions. 
\begin{definition}[Lipschitz regions]\label{def:regions}
A \emph{Lipschitz region} is an open set $D\subset\Omega$ whose boundary is Lipschitz. Throughout the paper, the sets $D_1,\ldots,D_K$ in \eqref{source} are Lipschitz regions, one for each interval $I_k$ of the prescribed time mesh: $D_k$ is the region where the source is active during $I_k$, and $p_k^\dagger=\chi_{D_k}$, regarded as an element of $X$, is the corresponding source.
\end{definition}
Since \eqref{inverse_problem} is posed in $X$, we recall this elementary and well-known fact: a Lipschitz region is determined by its characteristic function in $X$.
\begin{lemma}\label{lem:regions-determined}
Let $D$ and $D'$ be Lipschitz regions with $|D\triangle D'|=0$. Then $D=D'$.
\end{lemma}
\begin{proof}
Nonempty open sets have positive measure, so $|D\triangle D'|=0$ forces $\overline D=\overline{D'}$; since a Lipschitz region lies locally on one side of its boundary \cite[Definition~1.2.1.1]{Grisvard2011}, $D=\operatorname{int}\overline D=\operatorname{int}\overline{D'}=D'$.
\end{proof}
Consequently, recovering $p_k^\dagger$ in $X$ amounts to recovering $D_k$: for any representative $p$ of $p_k^\dagger$, the thresholded set $\{x\in\Omega:\ p(x)\ge\tfrac12\}$ differs from $D_k$ by a null set, and $D_k$ is the unique Lipschitz region with this property. This is the sense in which the thresholding step of Algorithm~\ref{alg:sequential} in Section~\ref{subsec:mainthm} returns the regions.

\subsection{Function spaces}
For an open set $G\subseteq\R^2$, $H^1(G)$ and $H^2(G)$ are the Sobolev spaces of functions in $L^2(G)$ whose weak derivatives of order up to one, respectively two, belong to $L^2(G)$, and $H^1_0(\Omega)$ is the closure of $C_c^\infty(\Omega)$ in $H^1(\Omega)$. For $0<s<1$, $H^s(G)$ is the fractional Sobolev space
\[
H^s(G):=\Bigl\{f\in L^2(G):\ \|f\|_{H^s(G)}^2:=\|f\|_{L^2(G)}^2+\int_G\int_G\frac{|f(x)-f(y)|^2}{|x-y|^{2+2s}}\dd x\dd y<\infty\Bigr\},
\]
see \cite[Section~2]{DiNezzaPalatucciValdinoci2012}. For such $s$, $H^{2+s}(G)$ is the space of $f\in H^2(G)$ whose second derivatives belong to $H^s(G)$. Given a bounded open set $D\subset\R^2$ with  Lipschitz boundary, $H^s(D)$ coincides with the space of restrictions to $D$ of functions in $H^s(\R^2)$ \cite[Theorem~5.4]{DiNezzaPalatucciValdinoci2012}.

For $0<s<1$, the Besov spaces $B^s_{2,\infty}(\mathbb R^2)$ and $B^s_{2,2}(\mathbb R^2)$ consist of all $f\in L^2(\mathbb R^2)$ such that, respectively,
\[
\sup_{|h|<1}|h|^{-s}\|f(\cdot+h)-f\|_{L^2(\mathbb R^2)}<\infty,
\qquad
\int_{|h|<1}|h|^{-2s}\|f(\cdot+h)-f\|_{L^2(\mathbb R^2)}^2\,\frac{\dd h}{|h|^2}<\infty,
\]
and $B^s_{2,2}(\mathbb R^2)=H^s(\mathbb R^2)$ with equivalent norms \cite[Section~2]{Sickel2021}.

\subsection{Compact and Hilbert--Schmidt operators}\label{subsec:operators}
We recall the facts on linear operators between Hilbert spaces $\mathcal H_1,\mathcal H_2$ that are used below; see \cite{Teschl2014} for proofs. The operator norm of a bounded linear operator $A:\mathcal H_1\to\mathcal H_2$ is written $\|A\|_{\mathcal L(\mathcal H_1,\mathcal H_2)}$, or $\|A\|_{\mathcal L(\mathcal H_1)}$ if $\mathcal H_2=\mathcal H_1$, or simply $\|A\|$ when the spaces are clearly understandable from the context. A bounded operator is compact if it is the limit in operator norm of finite-rank operators \cite[Section~5.2]{Teschl2014}. A compact operator $A$ has singular values $\sigma_1(A)\ge\sigma_2(A)\ge\cdots>0$, the square roots of the nonzero eigenvalues of $A^*A$ counted with multiplicity, and $\|A\|=\sigma_1(A)$ \cite[Theorem~6.7]{Teschl2014}. A bounded operator $A$ is Hilbert--Schmidt if
\[
\|A\|_{\rm HS}^2:=\sum_{n\ge1}\|Ae_n\|_{\mathcal H_2}^2<\infty
\]
for one, and then for every, orthonormal basis $\{e_n\}_{n\ge1}$ of $\mathcal H_1$ \cite[Lemma~6.11]{Teschl2014}. Hilbert--Schmidt operators are compact, $\|A\|_{\rm HS}^2=\sum_q\sigma_q(A)^2$, and hence $\|A\|\le\|A\|_{\rm HS}$. 

We also need the following fact on unbounded operators: a densely defined operator $A$ has an adjoint $A^*$, and $A$ is closable if and only if $\Dcal(A^*)$ is dense \cite[Section~2.2, Lemma~2.4]{Teschl2014}. 

\subsection{The Dirichlet Laplacian on the disk}\label{subsec:laplacian}
Let $\mathcal L=-\Delta$ with domain
$\Dcal(\mathcal L)=H^2(\Omega)\cap H_0^1(\Omega)$. This positive
self-adjoint operator has an orthonormal eigenbasis $\{\varphi_n\}_{n\ge1}$,
with eigenvalues counted according to multiplicity:
\begin{equation}\label{behaviour_of_eigenvalue}
0<\lambda_1\le\lambda_2\le\cdots\longrightarrow\infty.
\end{equation}
On the unit disk, $\lambda_n\sim4n$ by Weyl's law, and
\begin{equation}\label{eigenfunction_2d}
\varphi_n(r,\theta)
=\omega_nJ_{m_n}(\sqrt{\lambda_n}\,r)\,\Theta_n(\theta),
\qquad m_n\in\mathbb N_0.
\end{equation}
Here $J_m$ denotes the Bessel function of the first kind of order $m$
\cite[Chapter~10]{NIST2010}, $\sqrt{\lambda_n}$ is a positive zero of $J_{m_n}$,
the angular factor $\Theta_n(\theta)$ is $\cos(m_n\theta)$ or $\sin(m_n\theta)$ for $m_n\ge1$, and $\Theta_n\equiv1$ for $m_n=0$;
see \cite{GrebenkovNguyen:2013}. We fix the signs of the eigenfunctions by
\[
\omega_n=
\begin{cases}
\displaystyle\frac{\sqrt{2/\pi}}{J_{m_n+1}(\sqrt{\lambda_n})},&m_n\ge1,\\[6pt]
\displaystyle\frac{1}{\sqrt\pi J_1(\sqrt{\lambda_n})},&m_n=0.
\end{cases}
\]
The Bessel identities $J_m'(j)=-J_{m+1}(j)$ at a zero $j$ of $J_m$ and
$\int_0^1rJ_m(jr)^2\dd r=\tfrac12J_{m+1}(j)^2$ show that
$\|\varphi_n\|_X=1$; see \cite[Eqs.~10.6.2 and~10.22.37]{NIST2010}.
Bessel functions of different nonnegative integer orders have no common
positive zeros \cite{Siegel:2014}. Thus radial eigenvalues are simple,
whereas every nonradial eigenspace is spanned by its cosine and sine modes.

For $f\in X$ we write $f_n:=\ip{f}{\varphi_n}_X$ for its spectral coefficients, so that $f=\sum_nf_n\varphi_n$ and $\|f\|_X^2=\sum_n|f_n|^2$ by Parseval's identity. The spectral Galerkin spaces are
\begin{equation}\label{eq:galerkin-spaces}
V_n:=\operatorname{span}\{\varphi_1,\ldots,\varphi_n\},\qquad
\mathcal V_{\rm fin}:=\bigcup_{n\ge1}V_n,
\end{equation}
and $P_n:X\to V_n$ denotes the orthogonal projector. By Weyl's law,
\begin{equation}\label{eq:weyl-summability}
\sum_{n\ge1}\lambda_n^{-\kappa}<\infty
\qquad\text{for every }\kappa>1,
\end{equation}
which is used repeatedly below. For $z\in\partial\Omega$, the boundary coefficients of the eigenfunctions are
\begin{equation}\label{czn_2d}
c_{z,n}:=-\lambda_n^{-1}\partial_{\mathbf n}\varphi_n(z)
=\begin{cases}
\sqrt{2/\pi}\,\lambda_n^{-1/2}\,\Theta_n(\theta_z),&m_n\ge1,\\
\pi^{-1/2}\lambda_n^{-1/2},&m_n=0,
\end{cases}
\end{equation}
where the explicit form follows from \eqref{eigenfunction_2d} and $J_m'(j)=-J_{m+1}(j)$. This is the disk coefficient used in \cite{RundellZhang2020,LinZhangZhang2022}. For the sensor set $\Gamma_M$ of Section~\ref{section_intro}, put
\begin{equation}\label{eq:bn}
\mathbf c_n:=(c_{z_1,n},\ldots,c_{z_M,n})^\top\in\R^M,
\qquad \|\mathbf c_n\|_{\mathbb R^M}^2\le C_M\lambda_n^{-1},
\qquad C_M:=2M/\pi.
\end{equation}

The following logarithmic divergence of the sensor coefficients is the key reason of the unboundedness of the forward operator (Section~\ref{subsec:splitting}).
\begin{lemma}\label{lem:cn-divergence}
Let the sensors be pairwise distinct. For every $b\in\R^M$, the sum of $|\mathbf c_n\cdot b|^2$ over all modes of angular order $m_n\le N$ satisfies
\begin{equation}\label{eq:cn-b-divergence}
\sum_{n:\,m_n\le N}|\mathbf c_n\cdot b|^2=\frac{\|b\|_{\R^M}^2}{2\pi}\log N+O(1)
\qquad(N\to\infty).
\end{equation}
In particular, $(\mathbf c_n\cdot b)_{n\ge1}\in\ell^2$ if and only if $b=0$.
\end{lemma}
\begin{proof}
Write $j_{m,\ell}$ for the positive zeros of $J_m$ and use cosine/sine
pairs for each nonradial eigenspace. If
$T_m(b)=\sum_{j=1}^Mb_je^{im\theta_j}$, then \eqref{czn_2d} gives
\[
|\mathbf c_{m,\ell}^{\mathrm c}\cdot b|^2+
|\mathbf c_{m,\ell}^{\mathrm s}\cdot b|^2
=\frac2\pi j_{m,\ell}^{-2}|T_m(b)|^2.
\]
The first Rayleigh sum is
\[
\sum_{\ell\ge1}j_{m,\ell}^{-2}=\frac1{4(m+1)}.
\]
For completeness, it follows by comparing the coefficient of $z^2$ in
\[
\Gamma(m+1)(2/z)^mJ_m(z)
=\prod_{\ell\ge1}(1-z^2/j_{m,\ell}^2)
=1-\frac{z^2}{4(m+1)}+O(z^4);
\]
see \cite[Eqs.~10.2.2 and~10.21.15]{NIST2010}.
The nonradial contribution to the squared coefficient sum is consequently
$(2\pi)^{-1}\sum_{m\ge1}|T_m(b)|^2/(m+1)$. Its partial sums satisfy
\[
\sum_{m=1}^N\frac{|T_m(b)|^2}{m+1}
=\sum_{j,q=1}^Mb_jb_q\sum_{m=1}^N
\frac{e^{im(\theta_j-\theta_q)}}{m+1}
=\|b\|_{\mathbb R^M}^2\log N+O(1).
\]
Indeed, the diagonal terms give the logarithmic term. For $j\ne q$,
the sensors are distinct, so the partial sums of
$e^{im(\theta_j-\theta_q)}$ are bounded; summation by parts then bounds
the corresponding weighted sums. The radial contribution is the finite
quantity $(4\pi)^{-1}(\sum_jb_j)^2$. This proves \eqref{eq:cn-b-divergence}, and in particular $\sum_n|\mathbf c_n\cdot b|^2=\infty$ for $b\ne0$.
\end{proof}

\subsection{The Hilbert scale}\label{subsec:hilbert-scale}
For $s>0$, by standard  functional calculus~\cite[Theorems~3.2 and~3.6]{Teschl2014}, the self-adjoint operator $\mathcal L$
 defines the positive self-adjoint operator
\[
L_s:=(I+\mathcal L)^{s/2},\qquad
\Dcal(L_s)=\left\{f\in X:
\sum_{n\ge1}(1+\lambda_n)^s|f_n|^2<\infty\right\},
\]
with $L_s\varphi_n=(1+\lambda_n)^{s/2}\varphi_n$. The spaces $(\Dcal(L_s),\|L_s\cdot\|_X)$, $s>0$, form the Hilbert scale generated by $\mathcal L$; each is a Hilbert space, and $L_s$ is an isometric isomorphism from it onto $X$. Its inverse is the diagonal operator
\[
L_s^{-1}f=\sum_{n\ge1}(1+\lambda_n)^{-s/2}f_n\varphi_n,
\qquad f\in X,
\]
which is bounded with $\|L_s^{-1}\|_{\mathcal L(X)}=(1+\lambda_1)^{-s/2}$.
It is also compact in the sense of Section~\ref{subsec:operators}: truncating the series to $n\le N$ gives finite-rank
operators whose distance to $L_s^{-1}$ in operator norm equals
$(1+\lambda_{N+1})^{-s/2}$, which tends to zero by
\eqref{behaviour_of_eigenvalue}. For small orders it holds that:
\begin{equation}\label{eq:Ls-domain-Hs}
\Dcal(L_s)=H^s(\Omega)\ \text{with equivalent norms},\qquad 0<s<\tfrac12,
\end{equation}
see \cite[Definition~2.1 and Theorem~2.2]{Grubb2016}. In particular, for $0<s<\tfrac12$ the domain $\Dcal(L_s)$ is dense in $X$, because $H^s(\Omega)\supseteq H^1(\Omega)\supseteq C_c^\infty(\Omega)$ \cite[Proposition~2.2]{DiNezzaPalatucciValdinoci2012}.

The unknown sources $\{p^\dagger_k\}_{k=1}^K$ are characteristic functions of Lipschitz regions (Definition~\ref{def:regions}), and no characteristic function of a set of positive finite measure belongs to $H^{1/2}(\R^2)$ \cite[Theorem~1]{Sickel2021}. However, the next lemma shows that every smaller order is compatible, so that by \eqref{eq:Ls-domain-Hs} the Hilbert-scale penalty $\|L_sp\|_X$ is finite on the sources whenever $0<s<\tfrac12$.

\begin{lemma}\label{lem:Lsdomain}
If $D\subset\Omega$ is a Lipschitz region, then
$\chi_D\in\Dcal(L_s)$ for every $0<s<1/2$. The same conclusion holds
for finite linear combinations of such characteristic functions.
\end{lemma}
\begin{proof}
Extend $\chi_D$ by zero to $\mathbb R^2$. For a bounded open set with
Lipschitz boundary, $\chi_D\in B^{1/2}_{2,\infty}(\mathbb R^2)$;
this is classical, see \cite[Theorem~2 and Corollary~1]{Sickel2021} and
\cite[Proposition~4.3]{YuanSickelYang2025}. The elementary embedding
$B^{1/2}_{2,\infty}(\mathbb R^2)\hookrightarrow B^{s}_{2,2}(\mathbb R^2)=H^s(\mathbb R^2)$
for $0<s<1/2$ \cite[Section~2]{Sickel2021} and restriction to $\Omega$ give
$\chi_D\in H^s(\Omega)$, hence $\chi_D\in\Dcal(L_s)$ by \eqref{eq:Ls-domain-Hs}.
Linearity proves the last assertion.
\end{proof}

From now on we fix $s\in(0,\tfrac12)$. By Lemma~\ref{lem:Lsdomain}, this is the range in which the characteristic sources belong to $\Dcal(L_s)$, and every statement below holds for each such $s$.

\section{Sequential reconstruction:  analysis and convergence}\label{section_main}

This section is devoted to study the forward opertor and the regularization of its associated inverse problem. We first provide a  block lower triangular representation of the 
forward operator  with respect to the time intervals: on
each interval, the contribution of the current source separates from those of
sources active earlier. We then establish
nonclosability on $L^2(\Omega)$, introduce source-side compactification, and
prove uniqueness and convergence of the sequential regularized reconstruction.

\subsection{The forward operator and its representation on each time interval}\label{subsec:forwardop}

Given a generic source sequence $(p_1,\ldots,p_K)\in\Dcal(L_s)^K$, let $u$ be the solution of \eqref{PDE_2d} with the source $\sum_{k=1}^Kp_k\chi_{I_k}$ in place of $P^\dagger$, and let $\mathcal F(p_1,\ldots,p_K)$ be the vector of its fluxes at the sensors,
\[
\mathcal F:\Dcal(L_s)^K\to Y_T,
\qquad
\mathcal F(p_1,\ldots,p_K)(t)
:=\bigl(\partial_{\mathbf n}u(z_1,t),\ldots,\partial_{\mathbf n}u(z_M,t)\bigr)^\top\in\R^M,
\qquad 0<t<T.
\]
That the fluxes are well defined and belong to $Y_T$ is part of Lemma~\ref{lem:history-decomposition} below.

The measured data $g$ in \eqref{data} are instead the fluxes produced by the exact source \eqref{source}, that is,
\[
g=\mathcal F(p_1^\dagger,\ldots,p_K^\dagger),
\qquad p_k^\dagger=\chi_{D_k},
\]
with $D_1,\ldots,D_K$ the unknown Lipschitz regions of Definition~\ref{def:regions}, and the inverse problem \eqref{inverse_problem} asks to invert $\mathcal F$ at $g$; by Lemma~\ref{lem:regions-determined}, this is the same as recovering the regions themselves.

The analysis is carried out in the eigenbasis of Section~\ref{subsec:laplacian}: every $f\in X$ is expanded as $f=\sum_nf_n\varphi_n$ with $f_n=\ip{f}{\varphi_n}_X$, and in particular the sources and the solution have the spectral coefficients
\begin{equation}\label{dkn}
p_{k,n}:=\ip{p_k}{\varphi_n}_X,
\qquad p_{k,n}^\dagger:=\ip{p_k^\dagger}{\varphi_n}_X,
\qquad u_n(t):=\ip{u(\cdot,t)}{\varphi_n}_X .
\end{equation}
The flux at a sensor $z$ is read off through the boundary coefficients \eqref{czn_2d}: formally,
$$
\partial_{\mathbf n}u(z,t)=\sum_nu_n(t)\,\partial_{\mathbf n}\varphi_n(z)=-\sum_n\lambda_nu_n(t)\,c_{z,n},
$$
and collecting the $M$ sensors with the vectors $\mathbf c_n$ of \eqref{eq:bn} gives
$$
\mathcal F(p_1,\ldots,p_K)(t)=-\sum_n\lambda_nu_n(t)\,\mathbf c_n.
$$
Lemma~\ref{lem:history-decomposition} below justifies this formula and makes its dependence on the sources explicit. 

For $f\in\Dcal(L_s)$, Cauchy--Schwarz, \eqref{eq:bn} and \eqref{eq:weyl-summability} give
\begin{equation}\label{eq:source-series-summability}
\sum_{n\ge1}\|\mathbf c_n\|_{\R^M}\,|f_n|
\le\left(C_M\sum_{n\ge1}(1+\lambda_n)^{-s}\lambda_n^{-1}\right)^{1/2}
\|L_sf\|_X<\infty,
\end{equation}
and by Lemma~\ref{lem:Lsdomain} this applies to the characteristic sources. Since the exponential factors below have modulus at most one for $t\ge0$, the series in the next definition converge absolutely and uniformly on the closed time windows whenever $f\in\Dcal(L_s)$; on $\mathcal V_{\rm fin}$ they reduce to finite sums.

\begin{lemma}\label{lem:damped-series}
Let $\tau>0$ and $\sigma\ge0$. For every $f\in X$,
\begin{equation}\label{eq:damped-series-summability}
\sum_{n\ge1}\bigl\|\mathbf c_nf_n\,e^{-\lambda_n(\cdot+\sigma)}\bigr\|_{Y_\tau}
\le\Bigl(\frac{C_M}{2}\sum_{n\ge1}\lambda_n^{-2}e^{-2\lambda_n\sigma}\Bigr)^{1/2}\|f\|_X<\infty .
\end{equation}
In particular, $\sum_n\mathbf c_nf_ne^{-\lambda_n(t+\sigma)}$ converges absolutely in $Y_\tau$.
\end{lemma}
\begin{proof}
Since $\int_0^\tau e^{-2\lambda_nt}\dd t\le(2\lambda_n)^{-1}$, \eqref{eq:bn} gives
$\|\mathbf c_nf_n\,e^{-\lambda_n(\cdot+\sigma)}\|_{Y_\tau}^2\le\frac{C_M}{2}\lambda_n^{-2}e^{-2\lambda_n\sigma}|f_n|^2$. Taking square roots and summing, Cauchy--Schwarz yields \eqref{eq:damped-series-summability}, whose right-hand side is finite by \eqref{eq:weyl-summability}.
\end{proof}

\begin{definition}[Forward operators and history operators]\label{current-slab-and-history-op}
Let $\tau>0$. The forward operator on a time interval of length $\tau$ is the linear operator
\begin{equation}\label{eq:Ftau-def}
F_\tau:\Dcal(L_s)\subset X\to Y_\tau,
\qquad
(F_\tau f)(t):=\sum_{n\ge1}\mathbf c_nf_n\,(e^{-\lambda_nt}-1),
\qquad 0<t<\tau .
\end{equation}
For $1\le i<k\le K$, the history operator, which carries the contribution of the source active on $I_i$ into the data on $I_k$, is the linear operator
\begin{equation}\label{eq:Hki-def}
H_{k,i}:X\to Y_{\tau_k},
\qquad
(H_{k,i}f)(t):=-\sum_{n\ge1}\mathbf c_nf_n\,(1-e^{-\lambda_n\tau_i})\,e^{-\lambda_n(t+t_{k-1}-t_i)},
\qquad 0<t<\tau_k,
\end{equation}
where the series converges absolutely in $Y_{\tau_k}$ for every $f\in X$, by Lemma~\ref{lem:damped-series} with $\sigma=t_{k-1}-t_i\ge0$ and the factors $1-e^{-\lambda_n\tau_i}\in(0,1)$.
\end{definition}

For $f\in\Dcal(L_s)$, the functions $F_\tau f$ and $H_{k,i}f$ are continuous on $[0,\tau]$ and on $[0,\tau_k]$, respectively, and in particular belong to $Y_\tau$ and $Y_{\tau_k}$. 

However, the two operators behave differently with respect to the norm of $X$. The history operators $H_{k,i}$ are bounded on all of $X$ by Lemma~\ref{lem:damped-series}, indeed Hilbert--Schmidt (Proposition~\ref{prop:history-compact}). The forward operator $F_\tau$ is unbounded on its dense domain $\Dcal(L_s)$ and admits no closed extension in $X$ (Lemma~\ref{lem:Ftau-unbounded} and Proposition~\ref{prop:F-unbounded}); the symbol $F_\tau$ always refers to this operator in $X$, that is, with $\Dcal(L_s)$ carrying the norm of $X$. The same map regarded on the Hilbert space $(\Dcal(L_s),\|L_s\cdot\|_X)$ of Section~\ref{subsec:hilbert-scale} is denoted by $F_{\tau,s}$; it is bounded (Proposition~\ref{prop:F-unbounded}(iii)), and it is through $F_{\tau,s}$ that the regularization is set up in Section~\ref{subsec:compactification}.

\begin{remark}[On the domain of $F_\tau$]\label{rem:Ftau-domain}
Let us  rewrite \eqref{eq:Ftau-def} as the following difference:
\begin{equation}\label{eq:Ftau-splitting}
(F_\tau f)(t)=\sum_{n\ge1}\mathbf c_nf_n\,e^{-\lambda_nt}-\sum_{n\ge1}\mathbf c_nf_n .
\end{equation}
The first series converges absolutely in $Y_\tau$ for every $f\in X$ by Lemma~\ref{lem:damped-series} with $\sigma=0$. Comparing partial sums in~\eqref{eq:Ftau-splitting}, and identifying vectors in $\R^M$ with time-constant elements of $Y_\tau$, we see that the series in \eqref{eq:Ftau-def} converges in $Y_\tau$ exactly on
\[
\Dcal_{\max}:=\Bigl\{f\in X:\ \sum_{n\ge1}\mathbf c_nf_n\ \text{converges in }\R^M\Bigr\}.
\]
This is a dense subspace of $X$, since it contains $\mathcal V_{\rm fin}$, but it is proper. Indeed, fix $b\in\R^M\setminus\{0\}$. If $\Dcal_{\max}=X$, then $\sum_nf_n(\mathbf c_n\cdot b)$ would converge for every $f\in X$. Then the uniform boundedness principle would then give $(\mathbf c_n\cdot b)_n\in\ell^2$, contradicting Lemma~\ref{lem:cn-divergence}.

The proofs of Lemma~\ref{lem:Ftau-unbounded} and Proposition~\ref{prop:F-unbounded}(i)--(ii) below also apply to the restriction of $F_\tau$ to $\mathcal V_{\rm fin}$. Consequently, the operator defined by \eqref{eq:Ftau-def} on any domain $\Dcal$ with $\mathcal V_{\rm fin}\subset\Dcal\subset\Dcal_{\max}$ is unbounded and not closable with respect to the norm of $X$. 

As natural choice we choose $\Dcal(F_\tau)=\Dcal(L_s)$; the inclusion $\Dcal(L_s)\subset\Dcal_{\max}$ follows from \eqref{eq:source-series-summability}. This domain contains the characteristic sources (Lemma~\ref{lem:Lsdomain}), and the same estimate gives absolute and uniform convergence in $t\in[0,\tau]$ for every $f\in\Dcal(L_s)$. Equipped with the norm $\|L_s\cdot\|_X$, this domain is the Hilbert space on which the regularization is set up in Section~\ref{subsec:compactification}.
\end{remark}

The following lemma establishes the block lower triangular representation of the forward map $\mathcal F$ announced in \eqref{eq:intro-decomposition}, separating the contribution of the current source from those of earlier sources.

\begin{lemma}\label{lemma_forward_operator}\label{lem:history-decomposition}
For a source sequence $(p_1,\ldots,p_K)\in\Dcal(L_s)^K$, the fluxes at the sensors are well defined, and
on the $k$-th interval they satisfy
\begin{equation}\label{forward_2d}
\begin{aligned}
\mathcal F(p_1,\ldots,p_K)(t_{k-1}+t)
={}&-\sum_{i<k}\sum_{n\ge1}\mathbf c_np_{i,n}
(1-e^{-\lambda_n\tau_i})e^{-\lambda_n(t+t_{k-1}-t_i)}\\
&+\sum_{n\ge1}\mathbf c_np_{k,n}(e^{-\lambda_nt}-1),
\qquad 0<t<\tau_k,
\end{aligned}
\end{equation}
where the series converge absolutely and uniformly on $[0,\tau_k]$; in particular, $\mathcal F(p_1,\ldots,p_K)\in Y_T$. Consequently, with $v_k\in Y_{\tau_k}$, $v_k(t):=g(t_{k-1}+t)$ for $0<t<\tau_k$, the exact data satisfy the block lower triangular system
\begin{equation}\label{eq:block-triangular}
\begin{pmatrix} v_1\\ v_2\\ \vdots\\ v_K\end{pmatrix}
=
\begin{pmatrix}
F_{\tau_1} & & & \\
H_{2,1} & F_{\tau_2} & & \\
\vdots & & \ddots & \\
H_{K,1} & H_{K,2} & \cdots & F_{\tau_K}
\end{pmatrix}
\begin{pmatrix} p_1^\dagger\\ p_2^\dagger\\ \vdots\\ p_K^\dagger\end{pmatrix},
\end{equation}
that is,
\begin{equation}\label{eq:history-decomposition}
v_k=\sum_{i<k}H_{k,i}p_i^\dagger+F_{\tau_k}p_k^\dagger,
\qquad k=1,\ldots,K,
\end{equation}
where the sum is empty for $k=1$ because $u(\cdot,0)=0$. Forward substitution in \eqref{eq:block-triangular} isolates the current source,
\begin{equation}\label{eq:shifted-data-problem}
F_{\tau_k}p_k^\dagger=\widetilde v_k,
\qquad \widetilde v_1:=v_1,
\qquad \widetilde v_k:=v_k-\sum_{i<k}H_{k,i}p_i^\dagger\quad(k\ge2).
\end{equation}
\end{lemma}
\begin{proof}
The spectral coefficients $u_n$ in \eqref{dkn} solve
\[
u_n'(t)+\lambda_nu_n(t)=\sum_i p_{i,n}\chi_{I_i}(t),
\qquad u_n(0)=0.
\]
Variation of constants gives, for $0<t<\tau_k$,
\[
u_n(t_{k-1}+t)=\frac1{\lambda_n}
\left[\sum_{i<k}p_{i,n}(1-e^{-\lambda_n\tau_i})
e^{-\lambda_n(t+t_{k-1}-t_i)}+p_{k,n}(1-e^{-\lambda_nt})\right].
\]
Taking the outward normal derivative of the finite spectral sums and using
\eqref{czn_2d} yields \eqref{forward_2d}. Estimate
\eqref{eq:source-series-summability} makes the flux series absolutely and
uniformly convergent, and justifies the finite sum over intervals and the time
integrations. The spectral sums for $u(\cdot,t_{k-1}+t)$ also converge
uniformly in $t\in[0,\tau_k]$ in $\Dcal(L_{s+2})$, which embeds continuously
into $H^{s+2}(\Omega)$ \cite[Theorem~2.2]{Grubb2016}.
Indeed, the preceding formula gives
$|u_n(t_{k-1}+t)|\le\lambda_n^{-1}\sum_{i\le k}|p_{i,n}|$ uniformly in
$t\in[0,\tau_k]$, and since $\lambda_n\ge\lambda_1>0$,
$(1+\lambda_n)^{s+2}\lambda_n^{-2}\le(1+\lambda_1^{-1})^2(1+\lambda_n)^s$.
Squaring the coefficient bound and applying Cauchy--Schwarz to the finite
sum over $i$ gives
\[
\sup_{0\le t\le\tau_k}\sum_n(1+\lambda_n)^{s+2}|u_n(t_{k-1}+t)|^2
\le(1+\lambda_1^{-1})^2\,K\sum_{i\le k}\sum_n(1+\lambda_n)^s|p_{i,n}|^2<\infty,
\]
since $p_i\in\Dcal(L_s)$. Applying the same estimate with the spectral
sums restricted to $n>N$ shows that the tails tend to zero in
$\Dcal(L_{s+2})$, uniformly in $t\in[0,\tau_k]$.
The trace theorem and the Sobolev embedding, applied in local boundary
coordinates, give the continuous maps
\[
H^{s+2}(\Omega)\xrightarrow{\;\partial_{\mathbf n}\;}
H^{s+1/2}(\partial\Omega)\hookrightarrow C(\partial\Omega),
\]
since $\partial\Omega$ is one-dimensional and $s+1/2>1/2$; see
\cite[Proposition~3.8 and Theorem~8.2]{DiNezzaPalatucciValdinoci2012}.
Consequently, the normal derivatives of the finite spectral sums converge
uniformly on the boundary to $\partial_{\mathbf n}u$, identifying the flux
series with the pointwise normal derivative at every sensor.
This also agrees with the representation in \cite{RundellZhang2020,LinZhangZhang2022}.
Substitution of $p_i=p_i^\dagger$ gives
\eqref{eq:block-triangular}, \eqref{eq:history-decomposition} and \eqref{eq:shifted-data-problem}.
\end{proof}

\subsection{Compactness of history operators and nonclosability of $F_\tau$}\label{subsec:splitting}

The history operators and the current-interval forward operator have different mapping properties
with respect to the norm of $X$. The former are compact operators on $X$, while the latter has no
closed extension in $X$.

\begin{proposition}\label{prop:history-compact}
For $1\le i<k\le K$, $H_{k,i}:X\to Y_{\tau_k}$ is a Hilbert--Schmidt operator.
\end{proposition}
\begin{proof}
For a basis vector $\varphi_n$ the series \eqref{eq:Hki-def} has a single term, and \eqref{eq:bn} gives
\begin{equation*}
\begin{aligned}
\|H_{k,i}\varphi_n\|_{Y_{\tau_k}}^2
&=\|\mathbf c_n\|_{\R^M}^2(1-e^{-\lambda_n\tau_i})^2
 e^{-2\lambda_n(t_{k-1}-t_i)}\int_0^{\tau_k}e^{-2\lambda_nt}\dd t\\
&\le\frac{C_M}{2}\lambda_n^{-2}e^{-2\lambda_n(t_{k-1}-t_i)},
\end{aligned}
\end{equation*}
which is summable in $n$ by \eqref{eq:weyl-summability}. Since $H_{k,i}$ is bounded on $X$ by Lemma~\ref{lem:damped-series}, the
Hilbert--Schmidt criterion \cite[Lemma~6.11]{Teschl2014} completes the proof.
\end{proof}

\begin{lemma}\label{lem:Ftau-unbounded}
Let the sensors be pairwise distinct and let $\tau>0$. The operator $F_\tau:\Dcal(L_s)\subset X\to Y_\tau$ of Definition~\ref{current-slab-and-history-op} is densely defined and unbounded.
\end{lemma}
\begin{proof}
The domain $\Dcal(L_s)$ is dense in $X$ (Section~\ref{subsec:hilbert-scale}). Fix $b\in\R^M$, $b\ne0$, and let $y\equiv b$ be the corresponding time-constant element of $Y_\tau$, $\|y\|_{Y_\tau}=\sqrt\tau\,\|b\|_{\R^M}$. For $f\in\mathcal V_{\rm fin}$, \eqref{eq:Ftau-def} gives
\[
\ip{F_\tau f}{y}_{Y_\tau}=\sum_nf_n\gamma_n\,(\mathbf c_n\cdot b),
\qquad
\gamma_n:=\int_0^\tau(e^{-\lambda_nt}-1)\dd t=\frac{1-e^{-\lambda_n\tau}}{\lambda_n}-\tau,
\]
and $|\gamma_n|\ge\tau/2$ whenever $\lambda_n\ge2/\tau$, which excludes only finitely many $n$. For $N\in\mathbb N$ let
\[
f_N:=\sum_{n=1}^{N}\gamma_n(\mathbf c_n\cdot b)\varphi_n\in V_N\subset\mathcal V_{\rm fin}\subset\Dcal(L_s).
\]
By Parseval's identity,
\[
\|f_N\|_X^2=\sum_{n=1}^{N}\gamma_n^2|\mathbf c_n\cdot b|^2
\ge\frac{\tau^2}{4}\sum_{n\le N:\ \lambda_n\ge2/\tau}|\mathbf c_n\cdot b|^2
\longrightarrow\infty
\qquad(N\to\infty),
\]
because $\sum_{n\ge1}|\mathbf c_n\cdot b|^2=\infty$ by Lemma~\ref{lem:cn-divergence} and only finitely many terms are dropped. Moreover $\ip{F_\tau f_N}{y}_{Y_\tau}=\sum_{n=1}^{N}\gamma_n^2|\mathbf c_n\cdot b|^2=\|f_N\|_X^2$, so Cauchy--Schwarz gives, for every $N$ with $f_N\ne0$,
\[
\frac{\|F_\tau f_N\|_{Y_\tau}}{\|f_N\|_X}
\ge\frac{\|f_N\|_X}{\|y\|_{Y_\tau}}
=\frac{\|f_N\|_X}{\sqrt\tau\,\|b\|_{\R^M}}
\longrightarrow\infty .
\]
Hence $F_\tau$ is unbounded.
\end{proof}

\begin{proposition}[Non closability of $F_\tau$]\label{prop:F-unbounded}
Let the sensors be pairwise distinct and let $\tau>0$, and consider the densely defined operator $F_\tau:\Dcal(L_s)\subset X\to Y_\tau$ of Definition~\ref{current-slab-and-history-op}.
\begin{enumerate}
\item[(i)] The domain of its adjoint is
\begin{equation}\label{eq:Ftau-adjoint-domain}
\Dcal(F_\tau^*)=\Bigl\{y\in Y_\tau:\int_0^\tau y(t)\dd t=0\in\mathbb R^M\Bigr\},
\end{equation}
which is not dense in $Y_\tau$.
\item[(ii)]  $F_\tau$ is not closable. 
\item[(iii)] The same map, regarded as an operator $F_{\tau,s}:(\Dcal(L_s),\|L_s\cdot\|_X)\to Y_\tau$ on the Hilbert space of Section~\ref{subsec:hilbert-scale}, is bounded.
\end{enumerate}
\end{proposition}
\begin{proof}
Let $y\in Y_\tau$ and put $b_y=\int_0^\tau y(t)\dd t$. For $f\in\Dcal(L_s)$, the series \eqref{eq:Ftau-def} converges uniformly on $[0,\tau]$, so termwise integration gives
\[
\ip{F_\tau f}{y}_{Y_\tau}=\sum_nf_n\beta_n(y),\qquad
\beta_n(y)=-\mathbf c_n\cdot b_y+\rho_n(y),\qquad
\rho_n(y)=\int_0^\tau\mathbf c_n\cdot y(t)e^{-\lambda_nt}\dd t.
\]
By the definition of the adjoint, $y\in\Dcal(F_\tau^*)$ if and only if $f\mapsto\sum_nf_n\beta_n(y)$ is bounded on $\Dcal(L_s)$ with respect to $\|\cdot\|_X$; by Parseval's identity, testing on $f=\sum_{n\le N}\beta_n(y)\varphi_n\in V_N$, this holds if and only if $(\beta_n(y))\in\ell^2$.
Cauchy--Schwarz yields
\[
|\rho_n(y)|^2\le\frac{C_M}{2}\lambda_n^{-2}\|y\|_{Y_\tau}^2.
\]
Thus $(\rho_n(y))\in\ell^2$, and it remains to determine when
$(\mathbf c_n\cdot b)\in\ell^2$ for $b\in\mathbb R^M$.

By Lemma~\ref{lem:cn-divergence}, $(\mathbf c_n\cdot b)\in\ell^2$ if and only if $b=0$, which proves
\eqref{eq:Ftau-adjoint-domain}.

The integration map $y\mapsto\int_0^\tau y(t)\dd t$ is bounded and onto
$\mathbb R^M$. Its kernel is closed of codimension $M$, and its orthogonal
complement consists of the time-constant functions; so $\Dcal(F_\tau^*)$
is not dense, which proves (i). The adjoint-domain criterion
\cite[Lemma~2.4(ii)]{Teschl2014} then gives nonclosability, which proves (ii). In particular, $F_\tau$ has no closed extension in $X\times Y_\tau$, and no bounded extension to all of $X$, since such an extension would be closed.
Finally, for $f\in\Dcal(L_s)$, \eqref{eq:Ftau-def} and \eqref{eq:source-series-summability} give
\[
\|F_{\tau,s}f\|_{Y_\tau}
\le\sqrt\tau\sum_{n\ge1}\|\mathbf c_n\|_{\R^M}\,|f_n|
\le C_{\tau,s,M}\|L_sf\|_X,
\qquad
C_{\tau,s,M}:=\Bigl(\tau C_M\sum_{n\ge1}(1+\lambda_n)^{-s}\lambda_n^{-1}\Bigr)^{1/2},
\]
so $F_{\tau,s}$ is bounded on $(\Dcal(L_s),\|L_s\cdot\|_X)$. This proves (iii).
\end{proof}

Proposition~\ref{prop:F-unbounded} shows a direct obstruction in applying  standard regularization theory for unbounded operators (see \cite{Groetsch1992spectral,groetsch2007stable}), which rely exactly on the closability of the forward map.

By Proposition~\ref{prop:F-unbounded}(i) and the graph characterization in \cite[proof of Lemma~2.4(ii)]{Teschl2014}, every nonzero time-constant datum $y\equiv b$ admits a sequence $f_j\in\Dcal(L_s)$ such that $f_j\to0$ in $X$ and $F_\tau f_j\to y$ in $Y_\tau$. Hence, for every $\alpha>0$, the Tikhonov functional
\begin{equation*}
p\longmapsto\tfrac12\|F_\tau p-y\|_{Y_\tau}^2+\tfrac\alpha2\|p\|_X^2,
\qquad p\in\Dcal(L_s),
\end{equation*}
has infimum zero but no minimizer: it tends to zero along $(f_j)$, whereas a zero value would force $p=0$ and then $y=0$.

To overcome such obstruction we exploit Proposition~\ref{prop:F-unbounded}-(iii). On the Hilbert space $(\Dcal(L_s),\|L_s\cdot\|_X)$ the operator is bounded, so measuring the source in this stronger norm, which is Hilbert-scale regularization, restores a well-posed Tikhonov problem. This is carried out in Section~\ref{subsec:compactification}.

\subsection{Source-side compactification}\label{subsec:compactification}

Since $L_s^{-1}$ is an isometric isomorphism from $X$ onto $(\Dcal(L_s),\|L_s\cdot\|_X)$ and $F_{\tau,s}$ is bounded on the latter by Proposition~\ref{prop:F-unbounded}(iii), the composition
\begin{equation}\label{eq:Ataus-def}
\mathcal A_{\tau,s}:=F_{\tau,s}L_s^{-1}:X\to Y_\tau,
\qquad
(\mathcal A_{\tau,s}\phi)(t)=\sum_{n\ge1}(1+\lambda_n)^{-s/2}\,\mathbf c_n\phi_n\,(e^{-\lambda_nt}-1),
\qquad 0<t<\tau,
\end{equation}
is a bounded operator defined on all of $X$. Since $L_s^{-1}V_n=V_n$, this change of variable preserves the spectral
approximation spaces.

This construction is the standard one of regularization in Hilbert scales; see
\cite[Section~8.4]{EnglHankeNeubauer1996} and all the references therein. There, the
penalty $\|L_sp\|_X$ is used for a bounded forward operator on $X$, and the
substitution $\phi=L_sp$ converts the Hilbert-scale problem into a standard
Tikhonov problem for $FL_s^{-1}$. In the present setting the same substitution
plays a second role. By Proposition~\ref{prop:F-unbounded}, $F_\tau$ has no
closed extension in $X\times Y_\tau$, so no choice of domain containing $\mathcal V_{\rm fin}$,
with the topology inherited from $X$, makes $F_\tau$ closable (Remark~\ref{rem:Ftau-domain}). We
therefore do not extend $F_\tau$; we change the norm on its domain, which
is exactly the passage from $F_\tau$ to $F_{\tau,s}$, and
$\mathcal A_{\tau,s}=F_{\tau,s}L_s^{-1}$ is the representation of
$F_{\tau,s}$ on $X$ through the isometry $L_s$. It acts on
the same sources as $F_\tau$ and produces the same fluxes, as
\eqref{eq:compactified-domain-identity} below makes precise; only the
norm in which the source is measured has changed. Compactness of
$\mathcal A_{\tau,s}$ then follows from the decay of $(1+\lambda_n)^{-s/2}$,
and the Tikhonov problem~\eqref{eq:intro-tikhonov-phi} for $\mathcal A_{\tau,s}$ with penalty
$\|\phi\|_X$ coincides with the Tikhonov problem~\eqref{eq:intro-tikhonov-Ls} for $F_\tau$ with the
Hilbert-scale penalty $\|L_sp\|_X$. This holds also after restriction to a Galerkin
space $V_n$, because $L_s^{-1}V_n=V_n$.

\begin{proposition}\label{prop:Ataus-HS}
For every $\tau>0$, the operator $\mathcal A_{\tau,s}:X\to Y_\tau$ is
Hilbert--Schmidt.
\end{proposition}
\begin{proof}
The operator is bounded by \eqref{eq:Ataus-def}, and its action on a basis vector satisfies
\[
\|\mathcal A_{\tau,s}\varphi_n\|_{Y_\tau}^2
=(1+\lambda_n)^{-s}\|\mathbf c_n\|_{\R^M}^2
\int_0^\tau(1-e^{-\lambda_nt})^2\dd t
\le C_M\tau(1+\lambda_n)^{-s}\lambda_n^{-1}.
\]
The upper bounds are summable by \eqref{eq:weyl-summability}, and the Hilbert--Schmidt criterion
\cite[Lemma~6.11]{Teschl2014} gives the claim.
\end{proof}

Directly from \eqref{eq:Ataus-def},
\begin{equation}\label{eq:compactified-domain-identity}
\mathcal A_{\tau,s}L_sp=F_{\tau,s}p=F_\tau p,
\qquad p\in\Dcal(L_s).
\end{equation}
The equation for each characteristic source is therefore
equivalent to
$\mathcal A_{\tau_k,s}\phi_k^\dagger=\widetilde v_k$, with
$\phi_k^\dagger=L_sp_k^\dagger\in X$.

\begin{proposition}\label{prop:injectivity}
If $\Gamma_M$ is admissible, then $\mathcal A_{\tau,s}:X\to Y_\tau$ is
injective for every $\tau>0$.
\end{proposition}
\begin{proof}
We use the Dirichlet-series argument and separation of angular modes in
\cite[Lemmas~3.4 and~3.5]{RundellZhang2020}.
Let $\mu_1<\mu_2<\cdots$ be the distinct eigenvalues, let
$E_p=\ker(\mathcal L-\mu_pI)$, and let $\Pi_p$ be the orthogonal projector
onto $E_p$. Define
\[
B_pv:=\sum_{\lambda_n=\mu_p}\mathbf c_n\ip{v}{\varphi_n}_X,
\qquad v\in E_p.
\]
For $\phi\in X$, grouping by eigenvalue gives
\[
(\mathcal A_{\tau,s}\phi)(t)=\sum_{p\ge1}a_p(e^{-\mu_pt}-1),
\qquad a_p=(1+\mu_p)^{-s/2}B_p\Pi_p\phi.
\]
Because every eigenspace has dimension at most two,
\[
\|a_p\|_{\R^M}\le\sqrt{2C_M}(1+\mu_p)^{-s/2}\mu_p^{-1/2}\|\Pi_p\phi\|_X.
\]
Cauchy--Schwarz and Weyl's law yield $\sum_p\|a_p\|_{\R^M}<\infty$.
Thus the series is uniformly convergent on $[0,\tau]$, and its derivative
$-\sum_p\mu_pa_pe^{-\mu_pt}$ converges locally uniformly in the complex
half-plane $\operatorname{Re}t>0$.

If $\mathcal A_{\tau,s}\phi=0$ in $Y_\tau$, continuity and differentiation
give $\sum_p\mu_pa_pe^{-\mu_pt}=0$ on $(0,\tau)$.
Analytic continuation gives the same identity for all $t>0$.
Multiplication by $e^{\mu_1t}$ and passage to the limit as $t\to\infty$ give $a_1=0$;
dominated convergence applies since $\sum_p\mu_p\|a_p\|_{\R^M}e^{-\mu_pt_0}<\infty$
for every $t_0>0$. Repetition shows $a_p=0$ for every $p$.

For a radial eigenspace, $B_p$ is nonzero and hence injective. For a
nonradial eigenspace of order $m\ge1$, its matrix is a nonzero scalar
multiple of
\[
\mathcal M_m=
\begin{pmatrix}
\cos(m\theta_1)&\sin(m\theta_1)\\
\vdots&\vdots\\
\cos(m\theta_M)&\sin(m\theta_M)
\end{pmatrix}.
\]
The admissible pair of sensor angles gives a minor with determinant
$\sin(m(\theta_b-\theta_a))\ne0$. Hence $B_p$ is injective on every
$E_p$. It follows that $\Pi_p\phi=0$ for all $p$, and therefore $\phi=0$.
\end{proof}

\subsection{Uniqueness}\label{subsec:uniqueness}

The preceding injectivity result implies uniqueness for the dynamic source
on every initial sequence of intervals. Related two-point uniqueness results
are established in \cite{RundellZhang2020,LiZhang:2020,LinOuZhangZhang2024}.

\begin{proposition}\label{proposition_uniqueness}
Assume that $\Gamma_M$ is admissible. Let $D_j$ and $\widetilde D_j$,
$j=1,\ldots,K$, be Lipschitz regions (Definition~\ref{def:regions}) on the
same time mesh, with sources $p_j^\dagger=\chi_{D_j}$ and
$\widetilde p_j^\dagger=\chi_{\widetilde D_j}$. If their boundary
fluxes agree on $\Gamma_M\times(0,t_k)$, then
$p_j^\dagger=\widetilde p_j^\dagger$ in $X$ and $D_j=\widetilde D_j$ for $j=1,\ldots,k$.
In particular, the data \eqref{data} uniquely determine all source functions
and their regions.
\end{proposition}
\begin{proof}
Set $q_j=p_j^\dagger-\widetilde p_j^\dagger$.
By Lemma~\ref{lem:Lsdomain}, $q_j\in\Dcal(L_s)$.
Subtracting the two representations \eqref{eq:block-triangular} yields the homogeneous block triangular system
\[
F_{\tau_j}q_j+\sum_{i<j}H_{j,i}q_i=0,\qquad j\le k,
\]
which we solve by forward substitution. For $j=1$, \eqref{eq:compactified-domain-identity} gives
$\mathcal A_{\tau_1,s}L_sq_1=0$, so Proposition~\ref{prop:injectivity}
implies $q_1=0$. If $q_i=0$ for $i<j$, the history terms vanish and the
same argument gives $q_j=0$. Induction proves that
$p_j^\dagger=\widetilde p_j^\dagger$ in $X$ for $j\le k$, and
Lemma~\ref{lem:regions-determined} gives $D_j=\widetilde D_j$.
\end{proof}

\subsection{Finite-dimensional Galerkin--Tikhonov regularization}\label{subsec:onestep}

Throughout this subsection and the next, assume that $\Gamma_M$ is admissible.
Fix an interval index $k$ and abbreviate $\tau:=\tau_k$; this is the problem for a single diagonal block of \eqref{eq:block-triangular}. Set
\begin{equation*}
p_k^\dagger:=\chi_{D_k},
\qquad
\phi_k^\dagger:=L_s p_k^\dagger\in X.
\end{equation*}
The exact shifted data are
\begin{equation*}
\widetilde v_k
:=v_k-\sum_{i=1}^{k-1}H_{k,i}p_i^\dagger
=F_\tau p_k^\dagger
=\mathcal A_{\tau,s}\phi_k^\dagger.
\end{equation*}
Given perturbed shifted data $\widetilde v_k^\delta\in Y_\tau$, suppose that
\[
\norm{\widetilde v_k^\delta-\widetilde v_k}_{Y_\tau}
\le\varepsilon(\delta).
\]
Here $\varepsilon(\delta)$ bounds the error in the shifted data; it may include errors from reconstructing sources on earlier intervals, as well as measurement noise. We define
\begin{equation}\label{finiteDTRPB}
\phi_{n,\alpha}^\delta
:=\operatorname*{argmin}_{\phi\in V_n}
\left\{
\frac12\norm{\mathcal A_{\tau,s}\phi-\widetilde v_k^\delta}_{Y_\tau}^2
+\frac\alpha2\norm{\phi}_X^2
\right\}.
\end{equation}
For every $\alpha>0$, this functional is strictly convex and coercive on $V_n$, so the minimizer exists and is unique. The reconstructed source is $L_s^{-1}\phi_{n,\alpha}^\delta$.

With $P_n$ the orthogonal projector onto $V_n$, the discretization error
\begin{equation*}
\gamma_n(\tau,s)
:=\norm{\mathcal A_{\tau,s}(I-P_n)}_{\mathcal L(X,Y_\tau)}
\end{equation*}
tends to zero, as follows immediately from the Hilbert--Schmidt property.
\begin{lemma}\label{lem:gamma-n}
For every $\tau>0$, one has $\gamma_n(\tau,s)\to0$ as $n\to\infty$.
\end{lemma}
\begin{proof}
By Proposition~\ref{prop:Ataus-HS},
\[
\gamma_n(\tau,s)^2
\le\sum_{m>n}\norm{\mathcal A_{\tau,s}\varphi_m}_{Y_\tau}^2
\longrightarrow0.
\]
\end{proof}

The following convergence result is a standard Galerkin--Tikhonov argument; see also \cite[Theorem~4.2.8]{groetsch1984theory}. We give a short proof for the parameter rule used below.
\begin{lemma}\label{lem:single-slab-convergence}
Let $\phi_{n,\alpha}^\delta$ solve \eqref{finiteDTRPB}. Suppose that, as $\delta\to0$,
\begin{equation}\label{eq:parameter-rule}
n(\delta)\to\infty,
\qquad
\alpha(\delta)\to0,
\qquad
\frac{\varepsilon(\delta)+\gamma_{n(\delta)}(\tau,s)}
{\sqrt{\alpha(\delta)}}\to0.
\end{equation}
Then
\begin{equation*}
\norm{\phi_{n(\delta),\alpha(\delta)}^\delta-\phi_k^\dagger}_X\to0,
\end{equation*}
and consequently
\begin{equation}\label{eq:f-convergence}
\norm{L_s^{-1}\phi_{n(\delta),\alpha(\delta)}^\delta-p_k^\dagger}_X\to0.
\end{equation}
\end{lemma}
\begin{proof}
Write $A:=\mathcal A_{\tau,s}$ and suppress the parameter dependence on $\delta$. Comparing the minimizer with $P_n\phi_k^\dagger$ gives
\[
\norm{A\phi_{n,\alpha}^\delta-\widetilde v_k^\delta}_{Y_\tau}^2
+\alpha\norm{\phi_{n,\alpha}^\delta}_X^2
\le
\bigl(\varepsilon+\gamma_n(\tau,s)\norm{\phi_k^\dagger}_X\bigr)^2
+\alpha\norm{P_n\phi_k^\dagger}_X^2.
\]
The parameter rule implies that the minimizers are bounded, their residuals tend to zero, and
\[
\limsup_{\delta\to0}\norm{\phi_{n,\alpha}^\delta}_X
\le\norm{\phi_k^\dagger}_X.
\]
Given any sequence $\delta_j\to0$, boundedness yields a subsequence,
not relabeled, such that
$\phi_{n,\alpha}^{\delta_j}\rightharpoonup\phi$ in $X$.
Since $\mathcal A_{\tau,s}$ is Hilbert--Schmidt and hence compact,
$\mathcal A_{\tau,s}\phi_{n,\alpha}^{\delta_j}\to\mathcal A_{\tau,s}\phi$
strongly in $Y_\tau$; together with the vanishing residuals and
$\widetilde v_k^{\delta_j}\to\widetilde v_k$, every weak limit therefore solves $A\phi=\widetilde v_k$.
By Proposition~\ref{prop:injectivity}, it equals $\phi_k^\dagger$.
Since the initial sequence was arbitrary, the whole family converges
weakly to $\phi_k^\dagger$. Weak lower semicontinuity and the preceding
limsup bound give
$\norm{\phi_{n,\alpha}^\delta}_X\to\norm{\phi_k^\dagger}_X$.
Together with weak convergence, this proves strong convergence in $X$.
Finally, $L_s^{-1}$ is bounded on $X$, which proves \eqref{eq:f-convergence}.
\end{proof}

\subsection{The main convergence theorem}\label{subsec:mainthm}

We now reconstruct the dynamic source in
\begin{equation*}
\mathcal F(p_1^\dagger,\ldots,p_K^\dagger)=g.
\end{equation*}
By Lemma~\ref{lem:history-decomposition}, the exact measurements satisfy the block triangular system \eqref{eq:block-triangular}, that is,
\begin{equation}\label{eq:exact-step-k}
v_k=\sum_{i=1}^{k-1}H_{k,i}p_i^\dagger+F_{\tau_k}p_k^\dagger,
\end{equation}
and the measured data satisfy
\begin{equation}\label{eq:measurement-noise}
\norm{v_k^\delta-v_k}_{Y_{\tau_k}}\le\delta,
\qquad k=1,\ldots,K.
\end{equation}
\begin{definition}[Sequential Galerkin--Tikhonov reconstruction]\label{def:sequential-reconstruction}
Set $\widetilde v_1^\delta:=v_1^\delta$. After reconstructing $\widehat p_1^\delta,\ldots,\widehat p_{k-1}^\delta$, define
\begin{equation}\label{eq:shifted-data-delta}
\widetilde v_k^\delta
:=v_k^\delta-\sum_{i=1}^{k-1}H_{k,i}\widehat p_i^\delta, \quad k\geq 2.
\end{equation}
For each interval, compute
\begin{equation}\label{eq:kth-tikhonov}
\begin{aligned}
\widehat\phi_k^\delta
&:=\operatorname*{argmin}_{\phi\in V_{n_k(\delta)}}
\left\{
\frac12\norm{\mathcal A_{\tau_k,s}\phi-\widetilde v_k^\delta}_{Y_{\tau_k}}^2
+\frac{\alpha_k(\delta)}2\norm{\phi}_X^2
\right\},\\
\widehat p_k^\delta&:=L_s^{-1}\widehat\phi_k^\delta.
\end{aligned}
\end{equation}
\end{definition}
This is forward substitution in the triangular system \eqref{eq:block-triangular}: the exact earlier sources are replaced by their reconstructions, and each diagonal block $F_{\tau_k}$ is inverted through the regularized compactified problem \eqref{eq:kth-tikhonov}. Definition~\ref{def:sequential-reconstruction} is the reconstruction method introduced in this paper. Algorithm~\ref{alg:sequential} summarizes it, together with the final thresholding step that recovers the regions; see \eqref{eq:region-convergence} below. On $V_{n_k}$, the minimization in \eqref{eq:kth-tikhonov} is a linear least-squares problem of dimension $n_k$, which the experiments of Section~\ref{section_num} solve through a singular value decomposition.

\begin{algorithm}[htbp]
\caption{Sequential Galerkin--Tikhonov reconstruction}\label{alg:sequential}
\begin{algorithmic}[1]
\Require noisy data $v_k^\delta\in Y_{\tau_k}$, $k=1,\ldots,K$; smoothing order $s\in(0,\tfrac12)$; regularization parameters $\alpha_k>0$ and Galerkin dimensions $n_k\in\mathbb N$, $k=1,\ldots,K$
\Ensure reconstructed sources $\widehat p_1^\delta,\ldots,\widehat p_K^\delta$ and regions $\widehat D_1^\delta,\ldots,\widehat D_K^\delta$
\For{$k=1,\ldots,K$}
  \State \textit{History subtraction:} $\widetilde v_k^\delta\gets v_k^\delta-\sum_{i=1}^{k-1}H_{k,i}\widehat p_i^\delta$ \Comment{empty sum for $k=1$}
  \State \textit{Regularized step:} $\widehat\phi_k^\delta\gets\operatorname*{argmin}_{\phi\in V_{n_k}}\bigl\{\tfrac12\|\mathcal A_{\tau_k,s}\phi-\widetilde v_k^\delta\|_{Y_{\tau_k}}^2+\tfrac{\alpha_k}2\|\phi\|_X^2\bigr\}$
  \State \textit{Change of variables:} $\widehat p_k^\delta\gets L_s^{-1}\widehat\phi_k^\delta$
  \State \textit{Thresholding:} $\widehat D_k^\delta\gets\{x\in\Omega:\widehat p_k^\delta(x)\ge\tfrac12\}$
\EndFor
\end{algorithmic}
\end{algorithm}

The effective error in the shifted data is
\begin{equation}\label{eq:eta-k}
\eta_k^\delta
:=\norm{\widetilde v_k^\delta-F_{\tau_k}p_k^\dagger}_{Y_{\tau_k}}.
\end{equation}
\begin{theorem}\label{thm:dynamic-convergence}
Suppose that, for every $k=1,\ldots,K$, the parameters satisfy \eqref{eq:parameter-rule} with $\tau=\tau_k$, $n=n_k$, $\alpha=\alpha_k$, and $\varepsilon(\delta)=\eta_k^\delta$. Then
\begin{equation*}
\norm{\widehat p_k^\delta-p_k^\dagger}_X\to0
\qquad\text{as }\delta\to0,
\qquad k=1,\ldots,K.
\end{equation*}
\end{theorem}
\begin{proof}
For the first interval, $\eta_1^\delta\le\delta$, so Lemma~\ref{lem:single-slab-convergence} gives convergence. For $k\ge2$, subtracting \eqref{eq:exact-step-k} from \eqref{eq:shifted-data-delta} yields
\begin{equation}\label{eq:effective-noise-basic}
\eta_k^\delta
\le\delta+
\sum_{i=1}^{k-1}\norm{H_{k,i}}_{\mathcal L(X,Y_{\tau_k})}
\norm{\widehat p_i^\delta-p_i^\dagger}_X.
\end{equation}
The history operators are bounded by Proposition~\ref{prop:history-compact}. Hence convergence on intervals $1,\ldots,k-1$ implies $\eta_k^\delta\to0$, and Lemma~\ref{lem:single-slab-convergence} gives convergence on interval $k$. Induction completes the proof.
\end{proof}

\begin{remark}\label{rem:parameter-choice}
The parameter sequences can be chosen successively: after fixing the earlier reconstructions, choose $\alpha_k(\delta)\to0$ slowly enough that $\eta_k^\delta/\sqrt{\alpha_k(\delta)}\to0$, and then choose $n_k(\delta)\to\infty$ so that $\gamma_{n_k(\delta)}(\tau_k,s)/\sqrt{\alpha_k(\delta)}\to0$. 
\end{remark}

The convergence also implies that the unknown regions of Definition~\ref{def:regions} are recovered in measure. Indeed, define the thresholded regions of Algorithm~\ref{alg:sequential},
\[
\widehat D_k^\delta
:=\{x\in\Omega:\widehat p_k^\delta(x)\ge\tfrac12\}.
\]
On $\widehat D_k^\delta\mathbin{\triangle}D_k$, the pointwise error has magnitude at least $1/2$. Therefore
\begin{equation}\label{eq:region-convergence}
\abs{\widehat D_k^\delta\mathbin{\triangle}D_k}
\le4\norm{\widehat p_k^\delta-\chi_{D_k}}_X^2
\longrightarrow0.
\end{equation}
Thus Algorithm~\ref{alg:sequential} recovers the regions $D_k$, which by Lemma~\ref{lem:regions-determined} are the actual unknowns of \eqref{inverse_problem}, up to sets whose measure vanishes with the noise.

\section{Error propagation and sensor stability}\label{section_stability}
In this section we study some quantitative properties of Algorithm~\ref{alg:sequential}.  We first check how errors in earlier reconstructions enter later shifted data, with exponential decay in the temporal variable. We then study how adding sensors affects  stability on fixed Galerkin spaces.

\subsection{Propagation of reconstruction errors}\label{section_quantitative}

The history operators, that is, the off-diagonal blocks of \eqref{eq:block-triangular}, quantify how errors from earlier reconstructions enter later shifted data during forward substitution.
\begin{proposition}\label{prop:history-decay-quantitative}
There exists $C_H>0$, depending only on the domain and the fixed sensor set, such that
\begin{equation}\label{eq:history-decay-quantitative}
\norm{H_{k,i}}_{\mathcal L(X,Y_{\tau_k})}
\le C_H e^{-\lambda_1(t_{k-1}-t_i)},
\qquad 1\le i<k\le K.
\end{equation}
Consequently, writing $E_i^\delta:=\norm{\widehat p_i^\delta-p_i^\dagger}_X$, one has
\begin{equation}\label{eq:effective-noise-propagation}
\eta_k^\delta
\le\delta+C_H\sum_{i=1}^{k-1}
e^{-\lambda_1(t_{k-1}-t_i)}E_i^\delta.
\end{equation}
\end{proposition}
\begin{proof}
The estimate in the proof of Proposition~\ref{prop:history-compact} gives
\[
\sum_{n\ge1}\norm{H_{k,i}\varphi_n}_{Y_{\tau_k}}^2
\le\frac{C_M}{2}
\sum_{n\ge1}\lambda_n^{-2}e^{-2\lambda_n(t_{k-1}-t_i)}
\le C_H^2e^{-2\lambda_1(t_{k-1}-t_i)},
\]
where $C_H^2:=(C_M/2)\sum_{n\ge1}\lambda_n^{-2}<\infty$. The operator norm is bounded by the Hilbert--Schmidt norm, proving \eqref{eq:history-decay-quantitative}. Combining this with \eqref{eq:effective-noise-basic} proves \eqref{eq:effective-noise-propagation}.
\end{proof}

\subsection{Sensor number and  stability}\label{section_sensor_stability}

We now study the dependence on the sensor count, for fixed $\tau>0$ and $s$. For a sensor set $\Gamma_M$, let $\mathcal A_M:X\to Y_{\tau,M}$ be the operator $\mathcal A_{\tau,s}$ of \eqref{eq:Ataus-def} built with $\Gamma_M$, where $Y_{\tau,M}:=L^2(0,\tau;\R^M)$ is the space $Y_\tau$ of \eqref{eq:Ytau} with the sensor count made explicit, normed by the unweighted Euclidean norm on $\R^M$. Componentwise,
\[
\mathcal A_M\phi=(E_1\phi,\ldots,E_M\phi)^\top,
\qquad
(E_j\phi)(t):=\sum_{n\ge1}(1+\lambda_n)^{-s/2}c_{z_j,n}\phi_n\,(e^{-\lambda_nt}-1),
\qquad 0<t<\tau,
\]
where $E_j:X\to L^2(0,\tau)$ is the component at the sensor $z_j$.

\begin{proposition}\label{thm:sensor-singular-monotonicity}\label{thm:galerkin-sensor-stability}
Suppose that the sensor sets are nested, $\Gamma_{M+1}=\Gamma_M\cup\{z_{M+1}\}$. Then the singular values, listed in nonincreasing order, satisfy
\[
\sigma_q(\mathcal A_{M+1})\ge\sigma_q(\mathcal A_M),
\qquad q\ge1.
\]
For a fixed Galerkin space $V_n$, let
\[
\sigma_{n,M}
:=\inf_{\substack{\phi\in V_n\\\norm{\phi}_X=1}}
\norm{\mathcal A_M\phi}_{Y_{\tau,M}}
\]
be the smallest singular value of the restriction of $\mathcal A_M$ to $V_n$. If $\mathcal A_M$ is injective on $V_n$, then $\sigma_{n,M}>0$ and
\[
\norm{\phi-\psi}_X
\le\frac1{\sigma_{n,M}}
\norm{\mathcal A_M\phi-\mathcal A_M\psi}_{Y_{\tau,M}},
\qquad\phi,\psi\in V_n.
\]
Moreover, $\sigma_{n,M+1}\ge\sigma_{n,M}$, so the  stability constant does not increase when a sensor is added.
\end{proposition}
\begin{proof}
Writing $G_M:=\mathcal A_M^*\mathcal A_M$, one has
\[
G_{M+1}=G_M+E_{M+1}^*E_{M+1}\ge G_M.
\]
The singular-value inequality follows from the min--max principle \cite[Corollary~4.13]{Teschl2014}. On $V_n$, positivity of $\sigma_{n,M}$ follows from injectivity and compactness of the unit sphere, and the inverse estimate follows from its definition. Finally,
\[
\norm{\mathcal A_{M+1}\phi}_{Y_{\tau,M+1}}^2
=\norm{\mathcal A_M\phi}_{Y_{\tau,M}}^2
+\norm{E_{M+1}\phi}_{L^2(0,\tau)}^2,
\]
so taking the infimum over the unit sphere of $V_n$ proves the last assertion.
\end{proof}

\section{Numerical experiments}\label{section_num}
\newcommand{\LowNoiseClassicalDice}{0.806}
\newcommand{\LowNoiseClassicalASSD}{17.470}
\newcommand{\LowNoiseClassicalHD}{22.157}
\newcommand{\LowNoiseBoundedDice}{0.929}
\newcommand{\LowNoiseBoundedASSD}{1.130}
\newcommand{\LowNoiseBoundedHD}{3.327}
\newcommand{\LowNoiseFourierDice}{0.916}
\newcommand{\LowNoiseFourierASSD}{1.199}
\newcommand{\LowNoiseFourierHD}{3.282}
\newcommand{\LowNoiseDNNDice}{0.958}
\newcommand{\LowNoiseDNNASSD}{0.590}
\newcommand{\LowNoiseDNNHD}{1.651}
\newcommand{\LowNoisePairedDiceGain}{0.0422}
\newcommand{\LowNoisePairedDiceLow}{0.0192}
\newcommand{\LowNoisePairedDiceHigh}{0.0727}
\newcommand{\LowNoiseFallbackCount}{3}
\newcommand{\CurrentClassicalDice}{0.682}
\newcommand{\CurrentClassicalASSD}{48.172}
\newcommand{\CurrentClassicalHD}{53.176}
\newcommand{\CurrentBoundedDice}{0.883}
\newcommand{\CurrentBoundedASSD}{1.956}
\newcommand{\CurrentBoundedHD}{5.513}
\newcommand{\CurrentFourierDice}{0.896}
\newcommand{\CurrentFourierASSD}{1.648}
\newcommand{\CurrentFourierHD}{4.580}
\newcommand{\CurrentDNNDice}{0.922}
\newcommand{\CurrentDNNASSD}{1.205}
\newcommand{\CurrentDNNHD}{3.452}
\newcommand{\CurrentPairedDiceGain}{0.0265}
\newcommand{\CurrentPairedDiceLow}{0.0102}
\newcommand{\CurrentPairedDiceHigh}{0.0430}
\newcommand{\CurrentFallbackCount}{3}
\newcommand{\CurrentTestIntervals}{24}
\newcommand{\CurrentTestParents}{8}
In this section we examine refinements of the sequential Galerkin--Tikhonov reconstruction of Algorithm~\ref{alg:sequential}, on sequences of three intervals with $0.1\%$ and with $1\%$ noise per interval. The refinement stages comprise a box-constrained reconstruction (Section~\ref{subsec:exp-settings}), a Fourier boundary approximation (Section~\ref{subsec:fourier-refinement}) and a pretrained contour-network correction (Section~\ref{subsec:contour-refinement}). In all cases, the history term is computed from the classical Galerkin--Tikhonov reconstructions; the refinements act only on the current source. Results are collected in Section~\ref{sec:final_results}.

\subsection{Discretization, sequential Tikhonov and bounded reconstruction}\label{subsec:exp-settings}

We take $T=3/4$ and $K=3$ intervals $I_k=[t_{k-1},t_k)$ with $t_k=k/4$, so that $\tau_k=1/4$. For the spectral discretization we write $\sqrt{\lambda_n}=j_{m_n,\ell_n}$, where $j_{m,\ell}$ is the $\ell$-th positive zero of $J_m$ as in the proof of Lemma~\ref{lem:cn-divergence}, and we retain the eigenfunctions~\eqref{eigenfunction_2d} with angular order $m_n\le m_{\max}:=16$ and radial index $\ell_n\le\ell_{\max}:=24$, both the cosine and the sine mode when $m_n\ge1$:
\[
 \mathcal N:=\{n\in\mathbb N:\ m_n\le m_{\max},\ \ell_n\le\ell_{\max}\},\qquad
 N:=\#\mathcal N=792,\qquad
 V_{\mathcal N}:=\operatorname{span}\{\varphi_n:\ n\in\mathcal N\}.
\]
Here the subscript of $V$ is a finite set of indices; with this reading, the Galerkin spaces of~\eqref{eq:galerkin-spaces} are $V_n=V_{\{1,\ldots,n\}}$. The truncation is by angular order and radial index rather than by eigenvalue, so $\mathcal N$ is not of the form $\{1,\ldots,n\}$: one has $V_{111}\subsetneq V_{\mathcal N}\subsetneq V_{2377}$, and in particular $V_{\mathcal N}\ne V_N$. The identification of the Hilbert-scale problem~\eqref{eq:intro-tikhonov-Ls} with the standard Tikhonov problem~\eqref{eq:intro-tikhonov-phi} on a Galerkin space uses only that the space is spanned by eigenfunctions, so that $L_s^{-1}$ maps it onto itself. This holds for $V_{\mathcal N}$ as well, and the discrete problem~\eqref{eq:current-tikhonov} below is the time-sampled counterpart of~\eqref{eq:kth-tikhonov} with $V_{n_k}$ replaced by $V_{\mathcal N}$. 

The reconstruction grid consists of the $256\times256$ equispaced points of $[-1,1]^2$, with spacing $h=2/255$; we write $\Omega_h$ for the grid points that lie in $\Omega$. A grid function is a function $p:\Omega_h\to\R$, extended by zero to the remaining grid points, and it is mapped to spectral coefficients by the Riemann sum
\[
 (\Pi_hp)_n:=h^2\sum_{x\in\Omega_h}p(x)\varphi_n(x)\approx\ip{p}{\varphi_n}_X,
 \qquad n\in\mathcal N,
\]
so that $\Pi_hp\in\R^N$ approximates the coefficients of $p$ with respect to the retained eigenfunctions.

Each interval is observed at the relative times $\nu\Delta t$, $\nu=0,\ldots,124$, with $\Delta t=1/500$, that is, at the absolute times $t_{k-1}+\nu\Delta t$. The $M=16$ sensors sit at the angles
\begin{equation*}
 \theta_j=\frac{2\pi(j-1)}{M}
       +\zeta\cos\!\left(\frac{2\pi(j-1)}{M}\right),
 \qquad \zeta=\frac{\pi\sqrt{2}}{32},\qquad j=1,\ldots,M,
\end{equation*}
so each interval contributes $125\cdot16=2{,}000$ measurements. Since $\theta_9-\theta_1=\pi-2\zeta$ and $(\pi-2\zeta)/\pi=1-\sqrt{2}/16\notin\Q$, the set $\Gamma_{16}$ is admissible in the sense of Section~\ref{section_intro}. 

Boldface denotes the vectors and matrices of the discretization. Sampling~\eqref{eq:Ftau-def} and~\eqref{eq:Hki-def} at these times, with measurements ordered first by time and then by sensor, gives matrices $\mathbf F_k,\mathbf H_{k,i}\in\R^{(125M)\times N}=\R^{2000\times792}$ acting on coefficient vectors in $\R^N$; they represent $F_{\tau_k}$ and $H_{k,i}$ on $V_{\mathcal N}$, respectively. Since the three intervals have the same length and the same relative sampling times, $\mathbf F_1=\mathbf F_2=\mathbf F_3$; the index only refers to the interval on which the matrix acts. With the coefficient vectors $\mathbf p_i^\dagger:=\Pi_h\chi_{D_i}$ of the grid samples of the true sources, the sampled data and shifted data are the discrete counterparts of~\eqref{eq:history-decomposition} and~\eqref{eq:shifted-data-delta},
\begin{equation}\label{eq:current-block-system}
 \mathbf v_k^\delta=\sum_{i<k}\mathbf H_{k,i}\mathbf p_i^\dagger+\mathbf F_k\mathbf p_k^\dagger+\mathbf e_k,
 \qquad
 \widetilde{\mathbf v}_k^\delta=\mathbf v_k^\delta-\sum_{i<k}\mathbf H_{k,i}\widehat{\mathbf p}_i^\delta,
\end{equation}
where $\mathbf e_k$ is the noise vector~\eqref{eq:current-noise} and $\widehat{\mathbf p}_i^\delta$ are the classical reconstructions below. 

Let $\mathbf W_s:=\operatorname{diag}\bigl((1+\lambda_n)^{-s/2}\bigr)_{n\in\mathcal N}$ be the matrix of $L_s^{-1}$ on $V_{\mathcal N}$. The classical reconstruction on interval $k$ is
\begin{equation}\label{eq:current-tikhonov}
 \widehat{\boldsymbol\phi}_k^\delta:=\operatorname*{argmin}_{\boldsymbol\phi\in\R^N}
 \Bigl\{\tfrac12\bigl\|\mathbf F_k\mathbf W_s\boldsymbol\phi-\widetilde{\mathbf v}_k^\delta\bigr\|_2^2
 +\tfrac\alpha2\|\boldsymbol\phi\|_2^2\Bigr\},\qquad
 \widehat{\mathbf p}_k^\delta:=\mathbf W_s\widehat{\boldsymbol\phi}_k^\delta,\qquad
 \widehat p_k^\delta:=\sum_{n\in\mathcal N}\widehat p_{k,n}^\delta\varphi_n,
\end{equation}
where $\widehat p_{k,n}^\delta$ is the entry of $\widehat{\mathbf p}_k^\delta$ with index $n$. The minimizer is evaluated through a thin singular value decomposition of $\mathbf F_k\mathbf W_s$, with the filter factors $\sigma_q/(\sigma_q^2+\alpha)$ applied to its singular values $\sigma_q$. In terms of $\widehat p_k^\delta$, the penalty is the Hilbert-scale term $\alpha\|L_s\widehat p_k^\delta\|_X^2=\alpha\sum_{n\in\mathcal N}(1+\lambda_n)^s(\widehat p_{k,n}^\delta)^2$ of~\eqref{eq:intro-tikhonov-Ls}.  Since $\Delta t\,\|\cdot\|_2^2$ is a Riemann sum for $\|\cdot\|_{Y_{\tau_k}}^2$, \eqref{eq:current-tikhonov} approximates the regularized step of Algorithm~\ref{alg:sequential} with $\alpha_k=\alpha\,\Delta t$ for every $k$. We use $s=0.49$ at both noise levels. At $1\%$ noise, $\alpha=10^{-11/2}\approx3.16\times10^{-6}$, calibrated in an earlier experiment on a different reference source and not adjusted to the present data; at $0.1\%$ noise, $\alpha=10^{-15/2}\approx3.16\times10^{-8}$, selected on the $24$ calibration sequences introduced below. As in Algorithm~\ref{alg:sequential}, only the classical vectors $\widehat{\mathbf p}_i^\delta$ enter the history subtraction in~\eqref{eq:current-block-system}.

\subsubsection{A bounded version}
Since the sources are characteristic functions, their values lie in $[0,1]$, and the classical reconstruction~\eqref{eq:current-tikhonov} does not exploit this. For the numerical experiments only, we therefore also compute a box-constrained variant that enforces the bound pointwise; it serves as the starting point of the boundary refinements below, and its convergence and stability analysis is beyond the scope of this paper. With the grid-to-data map $\mathbf F_k\Pi_h$ we compute
\begin{equation*}
 \widehat p_k^{\delta,\mathrm B}:=\mathop{\arg\min}_{p:\,\Omega_h\to[0,1]}
 \left\{\frac{\|\mathbf F_k\Pi_hp-\widetilde{\mathbf v}_k^\delta\|_2^2}{2d_0^2}
       +\frac{\alpha_Bh^2}{2}\sum_{x\in\Omega_h}p(x)^2\right\},
 \quad
 d_0\approx0.612,
\end{equation*}
where $d_0$ is a data scale fixed in an earlier experiment, $\alpha_B=10^{-4}$ at $1\%$ noise and $\alpha_B=10^{-13/2}\approx3.16\times10^{-7}$ at $0.1\%$ noise, the latter selected on the calibration sequences among the values for which the Newton iteration converges on every calibration sequence. This box-constrained spatial $L^2$ problem differs from~\eqref{eq:current-tikhonov} in both the penalty and the feasible set, and $\widehat p_k^{\delta,\mathrm B}$ is not obtained by truncating $\widehat p_k^\delta$ to $[0,1]$. We use a dual semismooth Newton method; see~\cite{WangWuYuTeo2024} for this approach to box-constrained parabolic optimal control problems.

The bounded reconstruction is used both as the initial guess for the Fourier coefficients and as the input to the network, whereas the unconstrained Tikhonov reconstruction is used to evaluate the history term that is subtracted from the data on the subsequent time slabs.

\subsubsection{Sources and noisy data}
On each interval the true region $D_k$ consists of one connected component, and this is known to the refinement stages; it is a disk or a simple polygon, hence a Lipschitz region in the sense of Definition~\ref{def:regions}. Its shape is chosen independently from circles, squares, equilateral triangles and straight-edged five-point stars. Position, area, orientation and family may change independently across intervals.

For each interval, let $\boldsymbol\xi_k\in\R^{2000}$ have independent standard normal components, and set
\begin{equation}\label{eq:current-noise}
 \mathbf e_k:=\rho\,\|\mathbf F_k\mathbf p_k^\dagger\|_2\,\frac{\boldsymbol\xi_k}{\|\boldsymbol\xi_k\|_2},
 \qquad \delta_k:=\|\mathbf e_k\|_2,\qquad \rho\in\{10^{-2},10^{-3}\},
\end{equation}
which is added to the complete noiseless data vector in~\eqref{eq:current-block-system}, including the contributions of earlier sources. The noise is thus $0.1\%$ or $1\%$ of the isolated current-source signal.

There are $24$ stopping-calibration sequences, $80$ training sequences with two independent noise realizations, $16$ validation sequences and $8$ test sequences, all disjoint; the test sequences comprise $24$ source instances, six per family, and the partition is the same at both noise levels. Table~\ref{tab:current-parameters} summarizes the discretization, data and reconstruction settings.

\begin{table}[!htbp]
\centering
\numtablestyle
\caption{Settings for the discretization, data partitions, and classical and bounded reconstructions.}
\label{tab:current-parameters}
\begin{tabularx}{\linewidth}{@{}l Y@{}}
\toprule
Domain and grid & Unit disk; $256^2$ grid, $h=2/255$; unit source intensity.\\
Source & One known component per interval; circle, square, triangle or five-point star.\\
Position, orientation and size & Area $[0.04,0.12]$; center radius $\le0.65$;  random orientation.\\
Intervals and observations & $T=3/4$; three intervals of length $1/4$; $\Delta t=1/500$; $16$ sensors; $125$ times and $2{,}000$ measurements per interval.\\
Spectral discretization & Angular orders $m\le m_{\max}=16$; radial indices $\ell\le\ell_{\max}=24$; $N=792$ modes.\\
Noise & A vector with independent standard normal components, scaled to $0.1\%$ or $1\%$ of the current-source signal norm.\\
Data partitions & $24$ calibration, $80$ training, $16$ validation and $8$ test sequences.\\
Classical Tikhonov & $s=0.49$; $\alpha\approx3.16\times10^{-6}$ at $1\%$, $\approx3.16\times10^{-8}$ at $0.1\%$.\\
Bounded spatial $L^2$ & $\alpha_B=10^{-4}$ at $1\%$, $\approx3.16\times10^{-7}$ at $0.1\%$; $d_0\approx0.612$; constraint $0\le p\le1$, zero outside disk.\\
\bottomrule
\end{tabularx}
\end{table}

\subsubsection{Reconstruction errors}\label{subsubsec:metrics}
For a grid function $p$ we define the estimated support $S(p):=\{x\in\Omega_h:\ p(x)\ge\tfrac12\}$, the grid analogue of the thresholded region $\widehat D_k^\delta$ of Algorithm~\ref{alg:sequential}, and we write $D_{k,h}:=\Omega_h\cap D_k$ for the true region of Definition~\ref{def:regions} sampled on the grid; the classical reconstruction $\widehat p_k^\delta$ is evaluated on $\Omega_h$ for this purpose. An estimated support $S$ is compared with $D_{k,h}$ through one overlap measure and two boundary distances, all standard in the comparison of reconstructed regions; see \cite{TahaHanbury2015} for a survey and \cite{MaierHeinEtAl2024} for current recommendations. With $\#$ denoting the number of grid points, the Dice coefficient \cite{Dice1945} is defined as
\[
 \operatorname{Dice}(S,D_{k,h})=\frac{2\,\#(S\cap D_{k,h})}{\#S+\#D_{k,h}}\in[0,1].
\]
 Being an area measure, it is insensitive to how a mismatch is distributed along the boundary, so we complement it with two surface distances. The boundary $\partial_hS$ of a grid set $S$ is the set of points of $S$ with at least one of their four nearest grid neighbours outside $S$, and $\operatorname{dist}(x,B):=\min_{y\in B}\|x-y\|_2$ is the Euclidean distance from a grid point $x$ to a nonempty set $B$ of grid points. The average symmetric surface distance \cite{HeimannEtAl2009}
\[
 \operatorname{ASSD}(S,D_{k,h})
 =\frac12\Bigl(\frac1{\#\partial_hS}\sum_{x\in\partial_hS}\operatorname{dist}(x,\partial_hD_{k,h})
 +\frac1{\#\partial_hD_{k,h}}\sum_{y\in\partial_hD_{k,h}}\operatorname{dist}(y,\partial_hS)\Bigr)
\]
is the mean of the two directed average boundary distances.

The Hausdorff distance \cite{HuttenlocherKlandermanRucklidge1993} replaces the two means by maxima and is therefore determined by the single worst boundary point; we also use the modified distance
\[
 \operatorname{HD95}(S,D_{k,h})
 =\max\Bigl\{Q_{0.95}\bigl(\operatorname{dist}(x,\partial_hD_{k,h}):\ x\in\partial_hS\bigr),\
 Q_{0.95}\bigl(\operatorname{dist}(y,\partial_hS):\ y\in\partial_hD_{k,h}\bigr)\Bigr\},
\]
where $Q_{0.95}$ is defined by ordering its arguments $a_{(1)}\le\cdots\le a_{(m)}$ and linearly interpolating as follows:
\[
 u=0.95(m-1),\qquad j=\lfloor u\rfloor,\qquad \omega=u-j,\qquad
 Q_{0.95}(a)=(1-\omega)a_{(j+1)}+\omega a_{(\min\{j+2,m\})}.
\]
This reduces the influence of isolated large boundary errors \cite{MaierHeinEtAl2024}. All boundary distances in the tables and figures are expressed in units of the grid spacing $h$.

\subsection{Fourier refinement}\label{subsec:fourier-refinement}
Trigonometric boundary parameterizations with residual-based stopping have precedent in inverse source problems~\cite{RundellZhang:2018,LinOuZhangZhang2024}. Dropping the interval index, the fitted boundary has a free center $c_F$ and polar radius
\begin{equation*}
 x_F(\vartheta)=c_F+r_F(\vartheta)(\cos\vartheta,\sin\vartheta),\qquad
 r_F(\vartheta)=a_0+\sum_{q=2}^{16}
             \big(a_q\cos q\vartheta+b_q\sin q\vartheta\big),
\end{equation*}
where $\vartheta$ is the polar angle about $c_F$ and $q$ indexes boundary harmonics. The parameter vector $\mathbf a=(c_F,a_0,a_2,b_2,\ldots,a_{16},b_{16})$ has $33$ entries. 

Let $D(\mathbf a)$ be the region enclosed by $x_F$ and $\mathbf p(\mathbf a)\in\R^N$ the vector of spectral coefficients $p_n(\mathbf a)$, $n\in\mathcal N$, of $\chi_{D(\mathbf a)}$. Since $-\Delta\varphi_n=\lambda_n\varphi_n$, the divergence theorem gives
\[
 p_n(\mathbf a)=\ip{\chi_{D(\mathbf a)}}{\varphi_n}_X
 =-\frac1{\lambda_n}\int_{\partial D(\mathbf a)}\partial_{\mathbf n_D}\varphi_n,
 \qquad n\in\mathcal N,
\]
with respect to arc length, where $\mathbf n_D$ is the outward unit normal to $D(\mathbf a)$; the response $\mathcal G_k(\mathbf a):=\mathbf F_k\mathbf p(\mathbf a)$ is thus computed from boundary integrals, without first representing $D(\mathbf a)$ on the spatial grid. A Levenberg--Marquardt continuation minimizes
\begin{equation*}
 \tfrac12\|\mathcal G_k(\mathbf a)-\widetilde{\mathbf v}_k^\delta\|_2^2
 +\tfrac{\beta}{2}\sum_{q=2}^{16}q^4(a_q^2+b_q^2),
 \qquad \beta=10^{-2},10^{-3},\ldots,10^{-12},
\end{equation*}
stage by stage in $\beta$, retaining only candidates with $\min_\vartheta r_F(\vartheta)>0.005$ and $\max_\vartheta|x_F(\vartheta)|<0.99$.

The stage is selected by a discrepancy principle. The shifted data contain the errors of the earlier classical reconstructions, so the relevant error level is the effective error $\eta_k^\delta$ of~\eqref{eq:eta-k}, which cannot be computed from the data (Remark~\ref{rem:parameter-choice}); we calibrate its size relative to $\delta_k$ on the $24$ calibration sequences. Let $\eta_k^{(c)}:=\|\widetilde{\mathbf v}_k^{\delta,(c)}-\mathbf F_k\mathbf p_k^{\dagger,(c)}\|_2$ and $\delta_k^{(c)}$ denote the effective error and the noise level of calibration sequence $c$. Arrange the $24$ ratios $\eta_k^{(c)}/\delta_k^{(c)}$ in increasing order, $R_{k,1}\le\cdots\le R_{k,24}$, and set
\[
 \varkappa_1:=1,\qquad
 \varkappa_k:=\max\{1,R_{k,23}\}\quad(k=2,3),\qquad
 \delta_{{\rm eff},k}:=\varkappa_k\delta_k.
\]
Thus the later intervals use the second largest ratio, while the first has no contribution from previous reconstruction errors. The resulting factors are $(\varkappa_1,\varkappa_2,\varkappa_3)\approx(1,\,1.073,\,1.180)$ at $1\%$ noise and $(1,\,1.094,\,1.227)$ at $0.1\%$. They describe a combined noise-and-history scale. We stop at the first feasible stage with $\|\mathcal G_k(\mathbf a)-\widetilde{\mathbf v}_k^\delta\|_2\le\delta_{{\rm eff},k}$. If none of the computed feasible approximations satisfies this criterion, the bounded reconstruction is retained without neural correction.

\subsection{Pretrained contour-network refinement}\label{subsec:contour-refinement}

The network corrects the fitted radius while keeping the center $c_F$ fixed. Its input is a periodic signal at $256$ angular nodes with eight channels. The inputs do not include the shape family, absolute position or observation vector. Figure~\ref{fig:dnn-postprocessor} shows the architecture, a circular residual U-Net~\cite{Ronneberger2015UNet} for periodic signals with widths $32/64/128/192$ and about $1.3\times10^6$ parameters, collectively denoted by $w$; a pointwise output map followed by a fixed Fourier projection onto the modes $2,\ldots,32$ produces the correction $d_w(\vartheta)$. The corrected radius is
\begin{equation*}
 r_w(\vartheta)=r_F(\vartheta)
 \exp\!\left[\rho_\star\tanh\!\left(\frac{d_w(\vartheta)+c_w}{\rho_\star}\right)\right],
 \qquad \rho_\star=\log5,\qquad
 \langle r_w^2\rangle=\langle r_F^2\rangle,
\end{equation*}
where $\langle\cdot\rangle$ is the mean over $1{,}024$ equispaced angles and the constant $c_w$ is determined by bisection.  We set $\rho_\star=\log 5$ as a fixed architectural parameter, allowing radius corrections by factors between $1/5$ and $5$ while preventing unbounded relative deformations. We sample the boundary at $512$ angles and convert it into a smooth approximation of the characteristic function on the pixel grid.

\begin{figure}[p]
\centering
\includegraphics[width=\linewidth]{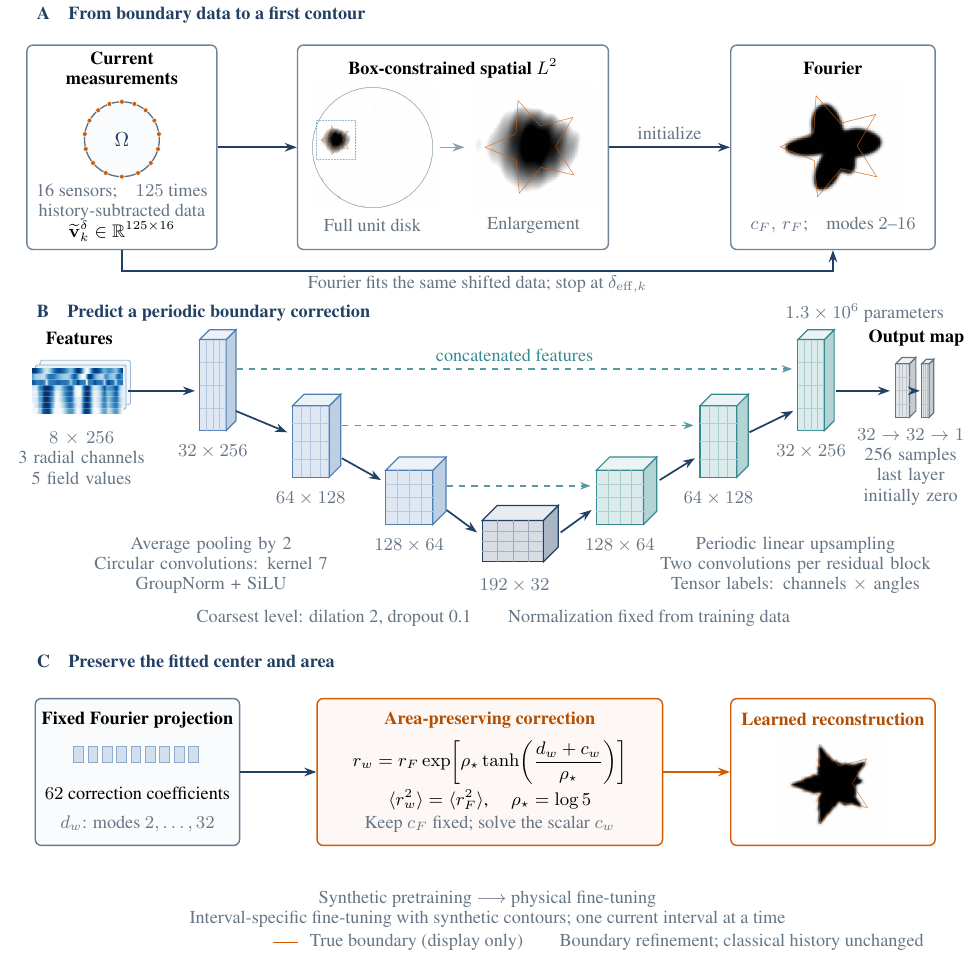}
\caption{Boundary refinement of the current source within the sequential Galerkin--Tikhonov reconstruction. The pictures show the first interval of one test sequence; all detail views share the crop marked on the disk, and orange contours show the truth for display only. Blue and teal tensors give encoder and decoder dimensions as channels times angular samples. The fixed Fourier projection is followed by a radial transformation that preserves the fitted center and the Fourier-estimated area.}
\label{fig:dnn-postprocessor}
\end{figure}

\subsubsection{Pretraining and interval-specific fine-tuning}
Training has three stages. First, $4{,}000$ synthetic perturbed contours, with $400$ more for validation, provide $10$ epochs of pretraining. Second, the network is fine-tuned for $100$ epochs on an earlier set of $80$ physical geometries with two noisy observations each and $16$ validation geometries, all disjoint from the present sequences. Third, three networks, one per interval, are initialized with the parameters obtained in the second stage and trained for $100$ epochs on the training inputs of their interval; each minibatch of $16$ source instances, with both noise realizations, is paired with an additional batch of $16$ synthetic contours, four per family, and the two losses are averaged with equal weights.  For either kind of input the loss is
\[
 \mathcal L=\mathcal L_{\rm surface}+0.1\,\mathcal L_{\rm BCE}+\mathcal L_{\log r},
\]
the sum of a quadrature approximation of the symmetric boundary distance in pixels, a balanced binary cross-entropy on a $128\times128$ object-centered grid, and a Smooth L1 loss between predicted and true log-radii. The trained parameters are selected by minimizing validation ASSD averaged equally over the shape families. The pretrained and fine-tuned weights, and the input normalization, are shared by the two noise levels.

\subsection{Overall results}\label{sec:final_results}
Tables~\ref{tab:current-reconstruction-comparison} and~\ref{tab:current-reconstruction-breakdown} summarize the sequential reconstruction and its refinements on all test sequences. At $0.1\%$ noise, the sequential reconstruction with Fourier and DNN refinement gives the best overall results, with similar accuracy across the three intervals, and improves the recovery of corners missed by Fourier approximation, especially in stars. The Fourier fit alone has lower mean Dice than the bounded reconstruction. At $1\%$ noise, the fully refined reconstruction still gives the best overall results, but later intervals and stars remain harder. The large classical ASSD at this noise level is mainly due to empty thresholded supports under the convention of Section~\ref{subsubsec:metrics}.

\begin{table}[!htbp]
\centering
\small
\caption{Sequential Galerkin--Tikhonov reconstruction and refinements of the current source on all $24$ test intervals from eight complete test sequences, at $0.1\%$ and at $1\%$ noise. All reconstructions use the classical Galerkin--Tikhonov estimates in the history term.}
\label{tab:current-reconstruction-comparison}
\begin{tabular}{lrrrrrr}
\toprule
 & \multicolumn{3}{c}{$0.1\%$ noise} & \multicolumn{3}{c}{$1\%$ noise} \\
\cmidrule(lr){2-4}\cmidrule(lr){5-7}
Reconstruction & Dice $\uparrow$ & ASSD $\downarrow$ & HD95 $\downarrow$ & Dice $\uparrow$ & ASSD $\downarrow$ & HD95 $\downarrow$ \\
\midrule
Sequential Galerkin--Tikhonov & 0.806 & 17.470 & 22.157 & 0.682 & 48.172 & 53.176 \\
with bounded $L^2$ reconstruction & 0.929 & 1.130 & 3.327 & 0.883 & 1.956 & 5.513 \\
with Fourier refinement & 0.916 & 1.199 & 3.282 & 0.896 & 1.648 & 4.580 \\
with Fourier and DNN refinement & 0.958 & 0.590 & 1.651 & 0.922 & 1.205 & 3.452 \\
\bottomrule
\end{tabular}
\end{table}
\begin{table}[!htbp]
\centering
\small
\caption{Sequential Galerkin--Tikhonov reconstruction with boundary refinement, by interval and true shape family, at $0.1\%$ and at $1\%$ noise: mean Dice, with the mean ASSD  in parentheses. Each interval contains eight cases; each family contains six.}
\label{tab:current-reconstruction-breakdown}
\begin{tabular}{lrrrrr}
\toprule
 & & \multicolumn{2}{c}{with Fourier} & \multicolumn{2}{c}{with Fourier and DNN} \\
\cmidrule(lr){3-4}\cmidrule(lr){5-6}
Group & Cases & $0.1\%$ & $1\%$ & $0.1\%$ & $1\%$ \\
\midrule
Interval 1 & 8 & 0.932 (0.96) & 0.914 (1.33) & 0.961 (0.56) & 0.965 (0.54) \\
Interval 2 & 8 & 0.911 (1.23) & 0.884 (1.72) & 0.952 (0.62) & 0.889 (1.63) \\
Interval 3 & 8 & 0.905 (1.41) & 0.889 (1.89) & 0.960 (0.60) & 0.912 (1.45) \\
\midrule
Circle & 6 & 0.996 (0.07) & 0.988 (0.18) & 0.996 (0.06) & 0.987 (0.19) \\
Square & 6 & 0.949 (1.03) & 0.933 (1.28) & 0.974 (0.52) & 0.951 (0.93) \\
Triangle & 6 & 0.929 (1.04) & 0.877 (1.98) & 0.964 (0.50) & 0.909 (1.42) \\
Star & 6 & 0.790 (2.65) & 0.784 (3.14) & 0.897 (1.28) & 0.841 (2.28) \\
\bottomrule
\end{tabular}
\end{table}

Evaluating the network adds negligible cost relative to the preceding stages. Figure~\ref{fig:current-sequence-s001} illustrates the sequential reconstruction and its refinements on one complete test sequence at both noise levels.

\begin{figure}[p]
\centering
\includegraphics[width=6.5in]{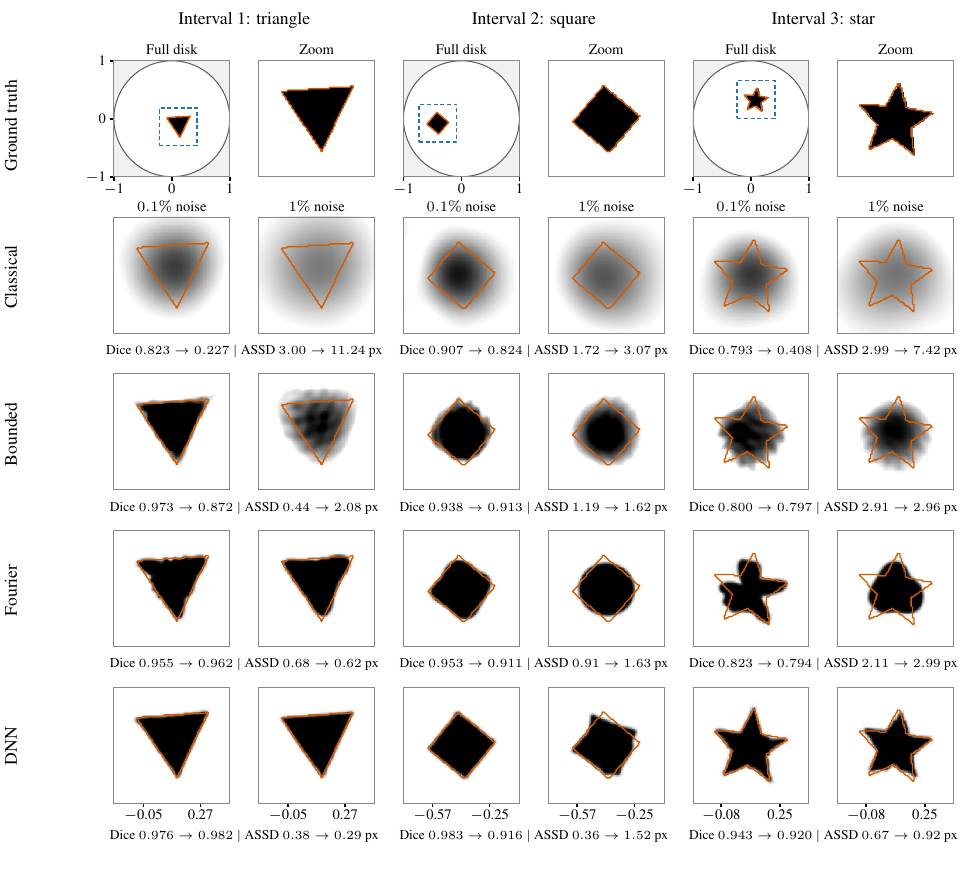}
\caption{Examples of sequential Galerkin--Tikhonov reconstruction and refinements for one test sequence (triangle, square, star) at $0.1\%$ noise (left) and $1\%$ noise (right).  The top row shows the truth and the enlargement used below; subsequent rows show the four reconstruction stages. Orange contours mark the truth. All reconstructions share coordinates and the intensity scale $[0,1]$. Labels report Dice and ASSD at $0.1\%\to1\%$.}
\label{fig:current-sequence-s001}
\end{figure}
\FloatBarrier

\section{Conclusions}

We developed a sequential method for reconstructing, from sparse boundary flux data, characteristic sources that are piecewise constant in time. The spectral representation shows that the forward operator $F_\tau$ is block lower triangular in time, separating the contribution of the current source from the history of earlier intervals, and the reconstruction proceeds by forward substitution. Although  $F_\tau$ is not closable on $L^2(\Omega)$, source-side smoothing produces a compact operator that is injective for admissible sensors. For $0<s<\tfrac12$, this gives finite-window uniqueness and convergence of the Galerkin--Tikhonov reconstruction under consistent choices of the regularization and discretization parameters. The error analysis quantifies how earlier reconstruction errors enter the data for a later interval, with exponential decay in time. 

In the numerical experiments, we used the sequential Galerkin--Tikhonov reconstruction as the starting approximation for Fourier and neural boundary refinement, in which the boundary of the source support is first parametrized in polar coordinates by a truncated Fourier series and then corrected by a neural network, thereby reducing the smoothing induced by the Hilbert-scale penalty. Future developments will aim to provide theoretical analysis of the box-constrained version of Algorithm~\ref{alg:sequential}, and to improve  the sequential Galerkin--Tikhonov reconstruction by using penalty terms represented by neural networks, following a network Tikhonov approach~\cite{bianchi2023uniformly,bianchi2025data,li2020nett,aleotti2026iterated}.

\FloatBarrier

\section*{Acknowledgments}
 Hongqin Zhang, Shupeng Zhang and Zhidong Zhang are supported by the National Key Research and Development Plan of China (Grant No. 2023YFB3002400). Davide Bianchi is supported by the Startup Fund of Sun Yat-sen University.

\bibliographystyle{abbrvurl}
\bibliography{ref}

\end{document}